\documentclass[12pt]{amsart}
\usepackage{amsmath}
\usepackage{amssymb}
\usepackage{longtable}
\usepackage[titletoc]{appendix}
\usepackage{color}
\usepackage{multirow,bigdelim}
\usepackage{stmaryrd}
\usepackage{verbatim}
\usepackage{tikz}
\usepackage{tikz-cd}
\usepackage{mathrsfs}
\usepackage{mathtools}
\usepackage[ruled,vlined,linesnumbered]{algorithm2e}

\usepackage{stmaryrd}
\SetSymbolFont{stmry}{bold}{U}{stmry}{m}{n} 
\usepackage{bm}
\usepackage{hyperref}

\allowdisplaybreaks

\newtheorem{theorem}[equation]{Theorem}
\newtheorem{lemma}[equation]{Lemma}

\newtheorem{example}[equation]{Example}

\theoremstyle{remark}
\newtheorem{remark}[equation]{Remark}

\numberwithin{equation}{subsection}

\allowdisplaybreaks[1]

\newcommand{\inorm}[1]{{\lvert #1 \rvert}_{\infty}}

\newcommand{\FF}{\mathbb{F}}

\newcommand{\TT}{\mathbb{T}}
\newcommand{\GG}{\mathbb{G}}

\newcommand{\CC}{\mathbb{C}}

\newcommand{\bA}{\mathbf{A}}

\newcommand{\bx}{\mathbf{x}}

\newcommand{\bu}{\mathbf{u}}
\newcommand{\bv}{\mathbf{v}}

\newcommand{\cL}{\mathcal{L}}

\DeclareMathAlphabet{\matheur}{U}{eur}{m}{n}

\newcommand{\fs}{\mathfrak{s}}

\DeclareMathOperator{\Div}{Div}
\DeclareMathOperator{\divisor}{div}

\DeclareMathOperator{\Ker}{Ker} \DeclareMathOperator{\GL}{GL}
\DeclareMathOperator{\Mat}{Mat}

\DeclareMathOperator{\Id}{Id} 
 
\DeclareMathOperator{\ord}{ord} \DeclareMathOperator{\wt}{wt}
  
\DeclareMathOperator{\Li}{Li}

\DeclareMathOperator{\nul}{null}

\newcommand{\oK}{\overline{K}}
\newcommand{\oL}{\overline{L}}

\newcommand{\tr}{\mathrm{tr}}

\newcommand{\SageMath}{\textsf{SageMath}}

\newcommand{\laurent}[2]{{#1 (\!( #2 )\!)}}

\begin{document}
\title[]{{\large{L\MakeLowercase{inear equations on $\MakeLowercase{t}$-modules}}}}

\author{Yen-Tsung Chen}
\address{Department of Mathematics, National Cheng Kung University, Tainan City, 70101, Taiwan(R.O.C.)}

\email{ytchen.math@gmail.com}

\author{Wei-Cheng Huang}
\address{Department of Mathematics, University of Rochester, Rochester, NY 14627, U.S.A.}

\email{w.huang@rochester.edu}

\author{Changningphaabi Namoijam}
\address{Department of Mathematics, Fairfield University, Fairfield, CT 06824, U.S.A.}

\email{cnamoijam@fairfield.edu}

\thanks{}

\subjclass[2010]{}

\date{\today}

\begin{abstract} 
    Let $F$ be a number field. Given finitely many $F$-valued points on a commutative algebraic group defined over $F$, a question of interest to number theorists is the determination of the group of their linear relations. In this article, we investigate an analogous problem in the $t$-module setting. Let $L$ be a global function field, and $E$ be a $d$-dimensional $t$-module defined over $L$. Given finitely many points on $E$ with entries in $L$, we establish the connection between their $\mathbb{F}_q[t]$-linear relations and polynomial solutions of Frobenius difference equations. 
    Consequently, we deduce an algorithm to compute the module of their $\mathbb{F}_q[t]$-linear relations. 
\end{abstract}


\maketitle


\section{Introduction}

    Let $G=\mathbb{G}_m(\mathbb{Q})=\mathbb{Q}^\times$ be the multiplicative group of the field of rational numbers. Given $\alpha_1,\dots,\alpha_\ell\in\mathbb{Q}^\times$, if there are $n_1,\dots,n_\ell\in\mathbb{Z}$, not all zero, such that
    \[
        \alpha_1^{n_1}\cdots\alpha_\ell^{n_\ell}=1,
    \]
    then $\alpha_1,\dots,\alpha_\ell$ are called \emph{multiplicative dependent}. The question of deciding the multiplicative dependency has connection to Diophantine geometry and transcendence theory. For example, given $\alpha_1,\dots,\alpha_\ell\in\mathbb{Q}^\times$ with $|\alpha_i|<1$ for $1\leq i\leq\ell$, the above multiplicative relation induces the following linear relation among logarithms
    \[
        n_1\log(\alpha_1)+\cdots+n_\ell\log(\alpha_\ell)=0,
    \]
    and the determination of the linear relations among logarithms of algebraic numbers is one of the core topics in transcendental number theory. The analogue of the multiplicative dependent question can be asked for a commutative connected algebraic group $G$ defined over a number field $F$. Given $P_1,\dots,P_\ell\in G(F)$, the $F$-valued points on $G$, how can we decide if there are $n_1,\dots,n_\ell\in\mathbb{Z}$, not all zero, such that $n_1P_1+\cdots n_\ell P_\ell=O$? Using techniques from the geometry of numbers, Masser \cite{Mas88} established an upper bound for the generators of the following relation group
    \begin{equation}\label{Eq:Intro_Relation_Module}
        \mathcal{R}:=\{(n_1,\dots,n_\ell)\in\mathbb{Z}^{\oplus\ell}\mid n_1P_1+\cdots+n_\ell P_\ell=O\}.
    \end{equation}
 Masser's result motivates our study of an analogous problem for Anderson $t$-modules, which are objects that play a crucial role in the arithmetic of global function fields.
    
\subsection{Main Result}
    Let $\mathbb{F}_q$ be a finite field with $q$ elements, where $q$ is a positive power of a prime number $p$. Let $\bA=\mathbb{F}_q[\theta]$ and let $K=\mathbb{F}_q(\theta)$, the field of fractions of $\bA$. For $f/g\in K$ with $f,g\in \bA$ and $g\neq 0$, we set $|f/g|_\infty:=q^{\deg_\theta f-\deg_\theta g}$ to be a normalized non-Archimedean norm on $K$. We denote the completion of $K$ with respect to $|\cdot|_\infty$ by $K_\infty$.  Note that $K_\infty$ can be identified with $\mathbb{F}_q(\!(1/\theta)\!)$. We further set $\mathbb{C}_\infty$ to be the the completion of a fixed algebraic closure of $K_\infty$, and denote by $\overline{K}$ the algebraic closure of $K$ in $\mathbb{C}_\infty$.

    For a field $L\subset\mathbb{C}_\infty$ containing $K$, we consider $L[t,\tau]$, the ring of twisted polynomials in two variables subject to the relations $\alpha t=t\alpha$, $t\tau=\tau t$, and $\alpha^q\tau=\tau\alpha$ for $\alpha\in L$. For $n\in\mathbb{Z}$, we define the $n$-fold twisting operation on $\mathbb{C}_\infty(\!(t)\!)$ by setting
    \begin{align*}
        \mathbb{C}_\infty(\!(t)\!)&\to\mathbb{C}_\infty(\!(t)\!)\\
        f=\sum_{i\geq i_0}c_it^i&\mapsto f^{(n)}:=\sum_{i\geq i_0}c_i^{q^n}t^i.
    \end{align*}
    This operation can be naturally extended to $\Mat_{r\times s}\big(\mathbb{C}_\infty(\!(t)\!)\big)$ entrywise, namely, for $B=(B_{ij})$ we define $B^{(n)}:=(B_{ij}^{(n)})$. Note that each element of $\Mat_{d\times d}(L[t,\tau])$ can be realized as an $\mathbb{F}_q$-linear endomorphism on the column space $\Mat_{d\times 1}(L[t])$. Indeed, given $\bv\in\Mat_{d\times 1}(L[t])$ and $A=\sum_{i=0}^NA_i\tau^i\in\Mat_{d\times d}(L[t,\tau])$ with $A_i\in\Mat_{d\times d}(L[t])$, we can define
    \begin{equation}\label{Eq:tau_action}
        A\cdot\bv:=\sum_{i=0}^NA_i\bv^{(i)}\in\Mat_{d\times 1}(L[t]).
    \end{equation}
    By an abuse of notation, we will drop the dot and simply write $A\bv$ when it is clear from the context.

    For $d>0$ and a field $L\subset\mathbb{C}_\infty$ containing $K$, a $d$-dimensional $t$-module defined over $L$ is a pair $E=(\mathbb{G}_a^d,\phi)$, where $\phi$ is an $\mathbb{F}_q$-algebra homomorphism $\phi:\mathbb{F}_q[t]\to\Mat_{d\times d}(L[\tau])$ satisfying some extra properties (see \S2.1 for  details). In particular, the structure map $\phi$ induces a non-scalar $\mathbb{F}_q[t]$-action on $\Mat_{d\times 1}(L)$. The study of $t$-modules dates back to Carlitz \cite{Car35} who introduced the $1$-dimensional $t$-module $\mathbf{C}=(\mathbb{G}_a,[\cdot])$ defined over $K$, which is uniquely determined by $[t]=\theta+\tau\in K[\tau]$. This $t$-module is now called the Carlitz module and can be regarded as an analogue of $\mathbb{G}_m$ in the global function field setting because of several striking parallel properties such as the analytic uniformization and the explicit class field generation \cite{Hay79}. Later, Drinfeld \cite{Dri76} introduced \emph{elliptic modules}, now called Drinfeld modules, as one of the tools for his solution of Langlands conjectures for $\GL_2$ over global function fields. Soon after, Anderson \cite{And86} found the correct path to generalize the theory to the higher dimensional setting by introducing the notion of abelian $t$-modules and $t$-motives.

    The main goal of the present article is to investigate an analogue of the multiplicative dependent problem for $t$-modules defined over a global function field. 
    Given a $d$-dimensional $t$-module $E=(\mathbb{G}_a^d,\phi)$ defined over $L$, we fix finitely many column vectors $P_1,\dots,P_\ell\in\Mat_{d\times 1}(L)$. An analogue of the multiplicative dependent problem in this context is the determination of the following $\mathbb{F}_q[t]$-module of relations
    \begin{equation}\label{Eq:Relation_Module}
        \mathcal{R}=\{(a_1,\dots,a_\ell)\in\mathbb{F}_q[t]^{\oplus\ell}\mid \phi_{a_1}(P_1)+\cdots+\phi_{a_\ell}(P_{\ell})=0\}.
    \end{equation}
   This problem has been studied by Denis for a wide family of $t$-modules, where an analogue of Masser's result was established using the theory of local height functions \cite{Den92a}, \cite{Den92b}.  The above multiplicative dependent problem and its $t$-motivic counterpart for iterated extensions of tensor powers of the Carlitz module play crucial roles in investigations of the linear relations among multiple zeta values in positive characteristic, and several new techniques and results have been introduced in \cite{Cha16,CPY19,CH25,KL16}. Adopting these $t$-motivic aspects, 
   Ho \cite{Ho20} 
   studied the multiplicative dependent problem for Drinfeld modules defined over $\bA$ with $P_1,\dots,P_\ell\in \bA$, and the first author \cite{Che23} generalized the result to Drinfeld modules defined over any global function field $L$ with $P_1,\dots,P_\ell\in L$.

    The innovation of the present article is a new approach to this multiplicative problem for $t$-modules without using their companion $t$-motives. In particular, our main result stated below provides an effective computational method to determine the module of $\mathbb{F}_q[t]$-linear relations for \emph{any} $t$-module defined over a global function field.
    
    \begin{theorem}\label{Thm:Main_Thm}
        Let $E=(\mathbb{G}_a^d,\phi)$ be a $d$-dimensional $t$-module defined over a global function field $L$. Given $P_1,\dots,P_\ell\in E(L)=\Mat_{d\times 1}(L)$, there is an explicitly constructed matrix $\mathbb{B}\in\Mat_{m\times n}(\mathbb{F}_q[t])$ with $\deg_t(\mathbb{B})\leq 1$, $n\geq\ell$ is an effectively computable number depending on $E,P_1,\dots,P_\ell,L$, and some $m\geq 0$ such that the following projection map is a well-defined surjective $\mathbb{F}_q[t]$-module homomorphism
        \begin{align*}
            \Ker(\mathbb{B})&\twoheadrightarrow\mathcal{R}\\
            (b_1,\dots,b_{n-\ell},a_1,\dots,a_\ell)^\tr&\mapsto(a_1,\dots,a_\ell),
        \end{align*}
        where $\Ker(\mathbb{B})$ is the collection of column vectors $\bm{x}\in\Mat_{n\times 1}(\mathbb{F}_q[t])$ with $\mathbb{B}\bm{x}=0$.
        In particular, there is an algorithm that determines the relation module $\mathcal{R}$.
    \end{theorem}

    The theoretical part will be the combinations of Theorem~\ref{Thm:SpPoly} and Theorem~\ref{Thm:Induced_Linear_System}. The algorithm will be given in \S4.

\subsection{Strategy and Organization}
    Let $\TT:=\{\sum_{i\geq 0}a_it^i\in\mathbb{C}_\infty\llbracket t\rrbracket\mid\lim_{i\to\infty}|a_i|_\infty=0\}$ be the Tate algebra on the closed unit disk of $\mathbb{C}_\infty$. Then, $\Mat_{d\times d}(\mathbb{C}_\infty[t,\tau])$ acts on $\Mat_{d\times 1}(\TT)$ in the same way as it acts on $\Mat_{d\times 1}(L[t])$ in \eqref{Eq:tau_action}. Given a $d$-dimensional $t$-module $E=(\mathbb{G}_a^d,\phi)$, consider the special functions of $\phi$
    \[
        \mathfrak{sf}_\phi:=\{\bm{f}\in\Mat_{d\times 1}(\TT)\mid(\phi_t-t)\bm{f}=0\}.
    \]
    It was demonstrated by Anderson \cite{And86} that if $E$ is an \emph{abelian} $t$-module, then $\mathfrak{sf}_\phi$ has a strong connection to the period lattice of $E$. In particular, Anderson established a new criterion of uniformizability for abelian $t$-modules using the so-called Anderson generating functions. Later, Maurischat extended this connection to \emph{any} $t$-modules \cite{Mau22} (see \cite[Thm.~A]{GM21} for the case of general Anderson modules). These results indicated that the Frobenius difference operator $\phi_t-t$ encodes important arithmetic information of the $t$-module $E$.

    The starting point of our approach is the observation that connects the multiplicative problem for a given $d$-dimensional $t$-module $E=(\mathbb{G}_a^d,\phi)$ defined over a global function field $L$ to the solution space of an inhomogeneous difference equation related to $\phi_t-t$. More precisely, the following result, restated in Theorem~\ref{Thm:SpPoly}, serves as the initial step of our investigation.
    \begin{theorem}\label{Thm:Intro_Initial}
        Let $E=(\mathbb{G}_a^d,\phi)$ be a $d$-dimensional $t$-module defined over a global function field $L$. Let $P_1,\dots,P_\ell\in\Mat_{d\times 1}(L)$ and let $\mathcal{R}$ be the module of $\mathbb{F}_q[t]$-linear relations given in \eqref{Eq:Relation_Module}. Define the solution space of the inhomogeneous difference equation
        \[
            \mathrm{Sol}_\phi(P_1,\dots,P_\ell):=\{(\bm{x},a_1,\dots,a_\ell)\in\Mat_{d\times 1}(L[t])\times\mathbb{F}_q[t]^{\oplus\ell}\mid(\phi_t-t)\bm{x}=\sum_{i=1}^\ell a_iP_i\}.
        \]
        If $P_1,\dots,P_\ell$ are $\mathbb{F}_q$-linearly independent, then the following map
        \begin{align*}
            \mathrm{Sol}_\phi(P_1,\dots,P_\ell)&\to\mathcal{R}\\
            (\bm{x},a_1,\dots,a_\ell)&\mapsto(a_1,\dots,a_\ell)
        \end{align*}
        is a well-defined $\mathbb{F}_q[t]$-module isomorphism.
    \end{theorem}
    Note that the above statement was first formulated in \cite[Remark~3.1.5]{Che23} for the case of Drinfeld modules, and was proved by using the companion $t$-comotives \footnote{Previously in the literature, a $t$-comotive was called a \emph{dual $t$-motive}. However, recent papers use the term $t$-comotive to avoid confusion with the dual of the $t$-motive in the categorical sense.} of Drinfeld modules. The approach adopted in the present paper 
    employ a direct study of the Frobenius difference operator $\phi_t-t$ without using the theory of $t$-comotives.

    With Theorem~\ref{Thm:Intro_Initial} in hand, we successfully transfer the multiplicative dependence problems for $t$-modules $E=(\mathbb{G}_a^d,\phi)$ to solving the inhomogeneous difference equation induced by $\phi_t-t$ and $P_1,\dots,P_\ell$. The second step of our strategy is to determine some essential fundamental results for matrices over the twisted polynomial ring $L[t,\tau]$. In particular, we establish an effective version of the LU decomposition (Theorem~\ref{Thm:Triangularization}) in our non-commutative setting. It then follows that $\mathrm{Sol}_\phi(P_1,\dots,P_\ell)$ can be embedded in the solution space of another difference equation where the difference operator $\phi_t-t$ is replaced by an upper triangular matrix in $\Mat_{d\times d}(L[t,\tau])$. This allows us to perform inductive arguments reducing the situation to the $1$-dimensional case. In particular, we prove that there is an explicitly constructed finite dimensional vector space $V$ over $\mathbb{F}_q$ with effectively computable dimension so that if $(\bm{x},a_1,\dots,a_\ell)\in\mathrm{Sol}_\phi(P_1,\dots,P_\ell)$ with $\bm{x}=(x_1,\dots,x_d)^\tr\in\Mat_{d\times 1}(L[t])$, then the coefficients of $x_i$,  for each $1\leq i\leq d$, as a polynomial in $L[t]$, in fact lie in $V$. After imposing this condition to our difference equation $(\phi_t-t)\bm{x}=\sum_{i=1}^\ell a_iP_i$, we deduce the desired matrix $\mathbb{B}$ and Theorem~\ref{Thm:Main_Thm} follows.
    
    The organization of the paper is as follows. In Section~2, we set up notation and recall the essential background on $t$-modules and divisors of curves defined over $\mathbb{F}_q$. In Section~3.1, we introduce the notion of special polynomials and make the connection between the relation module $\mathcal{R}$ and the solution space $\mathrm{Sol}_\phi(P_1,\dots,P_\ell)$. In Section~3.2, we transport some linear algebra tools such as the LU decomposition to the setting of matrices over twisted polynomial ring in two variables and establish an effective version. In Section~3.3, we demonstrate how to deduce the desired valuation constraint from the $1$-dimensional case. In Section~3.4, we combine the results proved in the previous section to construct the matrix $\mathbb{B}$ stated in Theorem~\ref{Thm:Main_Thm}. Finally, we describe our algorithm in Section~4 and provide some concrete examples and applications.

\section*{Acknowledgements}
The authors thank O\u{g}uz Gezm\.{i}\c{s}, Sheng-Yang Kevin Ho, and 
Federico Pellarin for helpful questions during the preparation of this paper. The authors also thank Andreas Maurischat for many suggestions that improved the clarity of the paper.  

\section{Notation and Preliminaries} \label{S:Prelim}
\begin{longtable}{p{2.25truecm}@{\hspace{5pt}$=$\hspace{5pt}}p{11truecm}}
$\FF_q$ & finite field with $q$ elements, where $q$ is a positive power of a prime number $p$. \\
$\bA$ & $\FF_q[\theta]$, the polynomial ring in $\theta$ over $\FF_q$. \\
$K$ & $\FF_q(\theta)$, the fraction field of $\bA$. \\
$L$ & a field extension of $K$ with $L\subset\mathbb{C}_\infty$. \\
$M_L$ & the set of places of $L$. \\
$K_\infty$ & $\laurent{\FF_q}{1/\theta}$, the completion of $k$ with respect to $\inorm{\,\cdot\,}$. \\
$\CC_{\infty}$ & the completion of an algebraic closure of $K_\infty$. \\
$\oK$ & the algebraic closure of $K$ inside $\CC_{\infty}$.  \\
\end{longtable}

For a field extension $L$ of $K$ with $L\subset\mathbb{C}_\infty$, we set $\oL$ to be the algebraic closure of $L$ in $\mathbb{C}_\infty$. Consider $L[t,\tau]$, the twisted polynomial ring in two variables subject to the relations $\alpha t=t\alpha$, $t\tau=\tau t$, and $\alpha^q\tau=\tau\alpha$ for $\alpha\in L$. We set 
\[
    \oL(t)(\!(\tau)\!):=\{\sum_{i=i_0}^\infty f_i\tau^{i}\mid f_i\in\oL(t),i_0\in\mathbb{Z}\}
\]
the twisted Laurent series field in $\tau$ with coefficients in $\oL(t)$. Note that $L[t,\tau]$ can be treated naturally as a subring of $\oL(t)(\!(\tau)\!)$. In this way, every non-zero element $f\in L[t,\tau]$ admits an inverse $f^{-1}\in\oL(t)(\!(\tau)\!)$. More precisely, if we denote by $\eta(f)\in\mathbb{Z}_{\geq 0}$ the largest non-negative integer such that $f\in L[t,\tau]\tau^{\eta(f)}$ and
\begin{align*}
    \partial:L[t,\tau]&\to L[t]\\
    f=f_0+f_1\tau+\cdots+f_s\tau^s&\mapsto\partial f:=f_0,
\end{align*}
then we can write $f=g\tau^{\eta(f)}$ for some $g\in L[t,\tau]$ with $g\not\in\Ker\partial$. If we further decompose $g=g_0+g_1\tau+\cdots+g_m\tau^m$ with $g_m\neq 0$, then \[g^{-1}=h=h_0+h_1\tau+\cdots\in\oL(t)(\!(\tau)\!)\] where $h_0=g_0^{-1}\in L(t)$ and for each $i\geq 1$, $h_i$ is defined recursively as
\[
    h_i=-g_0^{-1}\big(\sum_{j=1}^ig_jh_{i-j}^{(j)}\big)\in L(t).
\]
In particular,
\[
    f^{-1}=\tau^{-\eta(f)}g^{-1}=\sum_{i\geq 0}h_i^{\big(-\eta(f)\big)}\tau^{i-\eta(f)}\in\oL(t)(\!(\tau)\!).
\]

\subsection{Anderson \texorpdfstring{$t$}{t}-modules}

 For $d\in\mathbb{Z}_{>0}$, a $d$-dimensional $t$-module defined over $L$ is a pair $E=(\mathbb{G}_a^d,\phi)$ where $\phi$ is an $\mathbb{F}_q$-algebra homomorphism
    \begin{align*}
        \phi:\mathbb{F}_q[t]&\to\Mat_{d\times d}(L[\tau])\\
        a&\mapsto\phi_a
    \end{align*}
    such that $\partial\phi_t-\theta\Id_d\in\Mat_{d\times d}(L)$ is a nilpotent matrix. Here for $B=B_0+B_1\tau+\cdots+B_s\tau^s$ with $B_i\in\Mat_{d\times d}(L)$ we set $\partial B:=B_0\in\Mat_{d\times d}(L)$. 
    It has a companion left $L[t,\tau]$-module $\mathcal{M}_E:=\Mat_{1\times d}(L[\tau])$ called a $t$-motive, where the $L[\tau]$-module structure is given by the natural left multiplication while the $\mathbb{F}_q[t]$-action is uniquely determined by
    \[
        a\cdot\bm{m}:=\bm{m}\phi_a,~\bm\in\mathcal{M}_E.
    \]
    A $t$-module is called \emph{abelian} if $\mathcal{M}_E$ is finitely generated over $L[t]$. 
    There are many fundamental examples of abelian $t$-modules that have provided abundant properties for use in the study of the arithmetic theory for the category of abelian $t$-modules (see, for example, \cite{HJ20,NP24}).
    However, we emphasize that our results presented in this article work for any $t$-module including the non-abelian ones.

    In what follows, we give a few examples of $t$-modules, including an example of a non-abelian $t$-module.
    
    \begin{example}
        For $d=1$, let $\rho:\mathbb{F}_q[t]\to L[\tau]$ be the $1$-dimensional $t$-module uniquely determined by $\rho_t=\theta+\kappa_1\tau+\cdots+\kappa_r\tau^r$ with $r>0$, $\kappa_i\in L$, and $\kappa_r\neq 0$. Then $E=(\mathbb{G}_a,\rho)$ is called a Drinfeld module of rank  $r$. In particular, if $r=1$ and $\kappa_1=1$, then it is called the Carlitz module.
    \end{example}

    \begin{example}
        For $d>1$, consider $[\cdot]_d:\mathbb{F}_q[t]\to\Mat_{d\times d}(K[\tau])$ uniquely determined by
        \begin{align*}
            [t]_d=\begin{pmatrix}
                \theta & 1 & &\\
                 & \ddots & \ddots\\
                & & \theta & 1 \\
                \tau & &  &\theta
            \end{pmatrix}\in\Mat_{d\times d}(K[\tau]).
        \end{align*}
        Then we call $\mathbf{C}^{\otimes d}=(\mathbb{G}_a^d,[\cdot]_d)$ the $d$-th tensor powers of the Carlitz module.
    \end{example}

        \begin{example}\label{E:non-ab}

           The $2$-dimensional $t$-module $(\GG_a^2, \phi)$ uniquely determined by
            \[
                \phi_t:=\begin{pmatrix}
                    \theta+\tau & \tau\\
                    \theta & \theta
                \end{pmatrix} \in \Mat_{2\times 2}(K[\tau])
            \]
            is a non-abelian $t$-module \cite[Example 7.5]{Mau26b}. 
        \end{example}

    By an abuse of notation, we denote by $E(L)=\Mat_{d\times 1}(L)$ the $\mathbb{F}_q[t]$-module of $L$-valued points on the $d$-dimensional $t$-module $E=(\mathbb{G}_a^d,\phi)$ defined over $L$. For $P_1,\dots,P_\ell\in E(L)$ not all zero, the main focus of this paper is the determination of the $\mathbb{F}_q[t]$-module $\mathcal{R}$ of linear relations among $P_1,\dots,P_\ell$ defined in \eqref{Eq:Relation_Module}, which we will call the \emph{relation module}.

\subsection{Background on 
divisors}\label{SS:DivisorsBackground}

   Let $L$ be a finite extension of $K$ that corresponds to a smooth projective curve $X_L$ defined over $\mathbb{F}_q$. We denote by $g_L$ the genus of $X_L$. Let $M_L$ be the set of places of $L$. We denote by $\mathrm{Div}(L)$ the divisor group of $L$. Given $D\in\mathrm{Div}(L)$, we define
    \[
        \mathcal{L}(D):=\{f\in L^\times\mid\mathrm{div}(f)+D\geq 0\}\cup\{0\}, 
    \]
    where 
    \[
        \mathrm{div}(f):= \sum_{v\in M_L}\mathrm{ord}_v(f)\cdot(v)\in\mathrm{Div}(L).
    \]
    For $D=\sum_{v\in M_L}a_v\cdot(v)\in\mathrm{Div}(L)$, we define $\deg(D):=\sum_{v\in M_L}a_v\deg(v)\in\mathbb{Z}$.
    Note that $\deg\big(\mathrm{div}(f)\big)=0$ for any $f\in L^\times$.
    In addition, for any $D_1, D_2\in\mathrm{Div}(L)$ with $f\in\mathcal{L}(D_1)$ and $g\in\mathcal{L}(D_2)$, we have $\mathrm{div}(fg)=\mathrm{div}(f)+\mathrm{div}(g)$, and thus $fg\in\mathcal{L}(D_1+D_2)$. If we write $D_1=\sum_{v\in M_L}a_v\cdot(v)$ and $D_2=\sum_{v\in M_L}b_v\cdot(v)$, then their LCM is defined to be
    \[
        D_1\vee D_2:=\sum_{v\in M_L}\max\{a_v,b_v\}\cdot(v)\in\mathrm{Div}(L).
    \]
    
    For $F:=f_0+f_1t+\cdots+f_\ell t^\ell\in L[t]$, we set
    \[
        \mathrm{ord}_v(F):=\min_{i=0}^\ell\{\mathrm{ord}_v(f_i)\}.
    \]
    Given $B=(B_{ij})\in\Mat_{m\times n}(L[t,\tau])$ with $B_{ij}=B_{ij}^{[0]}+B_{ij}^{[1]}\tau+\cdots+B_{ij}^{[r]}\tau^r$, we set
    \[
        \mathrm{ord}_v(B):=\min_{i,j,k}\{\mathrm{ord}_v(B_{ij}^{[k]})\}
    \]
    and
    \[
        \divisor(B)=\sum_{v\in M_L}\mathrm{ord}_v(B)\cdot(v)\in\Div(L).
    \]
    Following Rosen \cite[page 47]{Ros02}, we decompose $\divisor(B)=\divisor(B)_0-\divisor(B)_\infty$ into the difference of two effective divisors, where
    \[
        \mathrm{div}(B)_0:=\sum_{v\in M_L}\max\{\mathrm{ord}_v(B),0\}\cdot(v)\in\mathrm{Div}(L).
    \]
    is the divisor of zeros and
    \begin{equation}\label{E:divinfty}
        \mathrm{div}(B)_\infty:=\sum_{v\in M_L}-\min\{\mathrm{ord}_v(B),0\}\cdot(v)\in\mathrm{Div}(L)
    \end{equation}
    is the divisor of poles associated to the matrix $B$.
    The height of the divisor $\mathrm{div}(B)_\infty$ is defined as $\mathrm{ht}(B)_\infty:=\deg\big(\mathrm{div}(B)_\infty\big)$.

\section{Linear relations among algebraic points on \texorpdfstring{$t$}{t}-modules}
    Let $L\subset\mathbb{C}_\infty$ be a finite extension of $K$ and $E=(\mathbb{G}_a^d,\phi)$ be a $d$-dimensional $t$-module defined over $L$. 
    Note that there is no additional assumption on $E$. 
    Given finitely many points $P_1,\dots,P_\ell\in E(L)=\Mat_{d\times 1}(L)$, recall their relation module $\mathcal{R}$ defined in \eqref{Eq:Relation_Module}.
    Since $\phi$ is an $\mathbb{F}_q$-algebra homomorphism, for any $(c_1,\dots,c_\ell)\in\mathbb{F}_q^{\oplus\ell}$ we must have
    \[
        c_1P_1+\cdots+c_\ell P_\ell=\phi_{c_1}(P_1)+\cdots+\phi_{c_\ell}(P_\ell).
    \]
    We call $\mathbb{F}_q$-linear relations among $P_1,\dots,P_\ell$ trivial relations on $E$. Throughout this section, we assume that $P_1,\dots,P_\ell$ has no trivial relations. 
    
\subsection{Special polynomials and the relation group}
 
   As defined in \cite{Mau22}, we set
    \[
        E[t]:=E(L)\otimes_{\mathbb{F}_q}\mathbb{F}_q[t]:=\left\{\sum_{i\geq 0}c_it^i\mid c_i\in E(L),~c_i=0~\mathrm{for}~i>\!\!>0\right\}.
    \]
    Then,
    $E[t]$ defines an $\mathbb{F}_q[t]\otimes_{\mathbb{F}_q}\mathbb{F}_q[t]$-module given by 
    \[
        (a\otimes b)\cdot\sum_{i\geq 0}c_it^i:=\sum_{i\geq 0}\phi_a(c_i)t^ib
    \]
    for any $a,b\in\mathbb{F}_q[t]$ and $\sum_{i\geq 0}c_it^i\in E[t]$. Note that the difference operator $\phi_t-t$ can be understood as $t\otimes 1-1\otimes t$, and defines an endomorphism on $E[t]$. 

    The natural exact sequence
    \[
        0\to\Ker(\phi_t-t)\to E[t]\overset{\phi_t-t}{\to}E[t]\to\mathrm{Coker}(\phi_t-t)\to 0
    \]
    provides an $\mathbb{F}_q[t]$-module
    \[
        \mathrm{Coker}(\phi_t-t)\cong E[t]/(\phi_t-t)E[t].
    \]
    By the natural inclusion $E(L)\hookrightarrow E[t]=E(L)\otimes_{\mathbb{F}_q}\mathbb{F}_q[t]$, we can treat $P_1,\dots,P_\ell$ as elements of $E[t]$. We set
    \[
        \mathbb{F}_q[t]\otimes 1:=\mathrm{Span}_{\mathbb{F}_q}\{t^i\otimes 1\mid i\geq 0\}
    \]
    and
    \[
        1\otimes\mathbb{F}_q[t]:=\mathrm{Span}_{\mathbb{F}_q}\{1\otimes t^i\mid i\geq 0\}.
    \]
    They are subrings of $\mathbb{F}_q[t]\otimes_{\mathbb{F}_q}\mathbb{F}_q[t]$ and each of them is isomorphic to $\mathbb{F}_q[t]$. Consider the $\mathbb{F}_q[t]\otimes 1$-module 
    \[
        \mathscr{M}_1:=\mathrm{Span}_{\mathbb{F}_q[t]\otimes 1}\{P_1,\dots,P_\ell\}\subset E[t]
    \]
   and the  $1\otimes\mathbb{F}_q[t]$-module
    \[
        \mathscr{M}_2:=\mathrm{Span}_{1\otimes \mathbb{F}_q[t]}\{P_1,\dots,P_\ell\}\subset E[t].
    \]
One can check directly that they generate the same submodule in $\mathrm{Coker}(\phi_t-t)$, namely
    \begin{equation}\label{Eq:Key_Id}
        \big(\mathscr{M}_1+(\phi_t-t)E[t]\big)/(\phi_t-t)E[t]=\big(\mathscr{M}_2+(\phi_t-t)E[t]\big)/(\phi_t-t)E[t].
    \end{equation}

    Now we consider the short exact sequence of $\mathbb{F}_q[t]$-modules
    \[
        0\to\mathscr{K}_1\to\mathbb{F}_q[t]^{\oplus\ell}\to\big(\mathscr{M}_1+(\phi_t-t)E[t]\big)/(\phi_t-t)E[t]\to 0,
    \]
   where 
    \[
        \mathscr{K}_1=\{(a_1,\dots,a_\ell)\in\mathbb{F}_q[t]^{\oplus\ell}\mid\phi_{a_1}(P_1)+\cdots+\phi_{a_\ell}(P_\ell)\in(\phi_t-t)E[t]\}.
    \]
    The following observation is immediate. Nevertheless, we include a quick proof for completion. 
    \begin{lemma}\label{Lem:Ker_1}
        For $\bm{f}\in E[t]$ with $\bm{f}\neq 0$, we have
        \[
            \deg_t\big((\phi_t-t)\bm{f}\big)=\deg_t(\bm{f})+1.
        \]
        It follows that 
        \[
            (\phi_t-t)E[t]\subset E[t]t
        \] 
        and $\phi_t-t:E[t]\to E[t]$ is an injective map. Moreover, we have
        \[
            \mathscr{K}_1=\mathcal{R}.
        \]
    \end{lemma}

    \begin{proof}
        For $\bm{f}\in E[t]$ with $\bm{f}\neq 0$, we have 
        \[
            \deg_t(\phi_t\bm{f})\leq\deg_t(\bm{f})<\deg_t(\bm{f})+1=\deg_t(t\bm{f}).
        \]
        Thus, the first assertion $\deg_t\big((\phi_t-t)\bm{f}\big)=\deg_t(\bm{f})+1$ now follows. It is now immediate that $(\phi_t-t)E[t]\subset E[t]t$ and that $\phi_t-t$ is injective on $E[t]$.
        
        To prove $\mathscr{K}_1=\mathcal{R}$, we observe that $\mathcal{R}\subset\mathscr{K}_1$ is clear from the definition. Thus, it remains to show that $\mathscr{K}_1\subset\mathcal{R}$. Let $(a_1,\dots,a_\ell)\in\mathbb{F}_q[t]^{\oplus\ell}$. On the one hand, we have
        \[
            \phi_{a_1}(P_1)+\cdots+\phi_{a_\ell}(P_\ell)\in E(L).
        \]
        On the other hand, by the definition of $\mathscr{K}_1$ and the property established in the previous paragraph, we have
        \[
            \phi_{a_1}(P_1)+\cdots+\phi_{a_\ell}(P_\ell)\in(\phi_t-t)E[t]\subset E[t]t.
        \]
        Note that $E[t]$ admits the direct sum decomposition $E[t]=E(L)\oplus E[t]t.$
        It follows that 
        \[
            \phi_{a_1}(P_1)+\cdots+\phi_{a_\ell}(P_\ell)\in E(L)\cap E[t]t=\{0\}.
        \]
        This verifies $\mathscr{K}_1\subset\mathcal{R}$, and hence completes the proof.
    \end{proof}

    We consider another short exact sequence of $\mathbb{F}_q[t]$-modules
    \[
        0\to\mathscr{K}_2\to\mathbb{F}_q[t]^{\oplus\ell}\to\big(\mathscr{M}_2+(\phi_t-t)E[t]\big)/(\phi_t-t)E[t]\to 0,
    \]
  where 
    \[
        \mathscr{K}_2=\{(a_1,\dots,a_\ell)\in\mathbb{F}_q[t]^{\oplus\ell}\mid a_1P_1+\cdots+a_\ell P_\ell\in(\phi_t-t)E[t]\}.
    \]
Define
    \begin{equation}\label{E:Sol(P)}
        \mathrm{Sol}_\phi(P_1,\dots,P_\ell):=\{(\bm{x},a_1,\dots,a_\ell)\in E[t]\times\mathbb{F}_q[t]^{\oplus\ell}\mid(\phi_t-t)\bm{x}=a_1P_1+\cdots+a_\ell P_\ell\}.
    \end{equation}
     We have the following observation.
     
    \begin{lemma}\label{Lem:Ker_2}
        The natural projection
        \begin{align*}
            \mathrm{Sol}_\phi(P_1,\dots,P_\ell)&\to\mathscr{K}_2\\
            (\bm{x},a_1,\dots,a_\ell)&\mapsto(a_1,\dots,a_\ell)
        \end{align*}
        is an isomorphism.
    \end{lemma}

    \begin{proof}
        Since the map in question is clearly well-defined and surjective, it remains to verify the injectivity. Note that the kernel of the map is given by the set
        \[
            \{(\bm{x},0,\dots,0)\in E[t]\times\mathbb{F}_q[t]^{\oplus\ell}\mid(\phi_t-t)\bm{x}=0\}\subset\mathrm{Sol}_\phi(P_1,\dots,P_\ell).
        \]
        However, by Lemma~\ref{Lem:Ker_1}, the map $\phi_t-t:E[t]\to E[t]$ is injective. Hence for $\bm{x}\in E[t]$ with $(\phi_t-t)\bm{x}=0$, we must have $\bm{x}=0$. This verifies that the map in question is also injective. The desired result now follows.
    \end{proof}

    In what follows, we introduce the notion of special polynomials for $\phi$ associated to $P_1,\dots,P_\ell$. More precisely, we define
    \[
        \mathfrak{sp}_\phi(P_1,\dots,P_\ell):=\mathscr{M}_2\cap(\phi_t-t)E[t].
    \]
    Now we are ready to prove the main result in this section, which shows that there is a one-to-one correspondence between non-trivial relations among $P_1,\dots,P_\ell$ on $E$ and their non-zero special polynomials.
    
    \begin{theorem}\label{Thm:SpPoly}
        Assume $P_1,\dots,P_\ell$ have no trivial relation in the sense that $P_1,\dots,P_\ell$ are linearly independent over $\mathbb{F}_q$. Then the following assertions hold.
        \begin{enumerate}
            \item We have the following $\mathbb{F}_q[t]$-module isomorphism.
                \begin{align*}
                    \rho_1:\mathrm{Sol}_\phi(P_1,\dots,P_\ell)&\to\mathcal{R}\\
                    (\bm{x},a_1,\dots,a_\ell)&\mapsto(a_1,\dots,a_\ell).
                \end{align*}
            \item We have the following $\mathbb{F}_q[t]$-module isomorphism.
                \begin{align*}
                    \rho_2:\mathrm{Sol}_\phi(P_1,\dots,P_\ell)&\to\mathfrak{sp}_\phi(P_1,\dots,P_\ell)\\
                    (\bm{x},a_1,\dots,a_\ell)&\mapsto(\phi_t-t)\bm{x}.
                \end{align*}
            \item We have the following $\mathbb{F}_q[t]$-module isomorphism.
                \begin{align*}
                    \rho_3:=\rho_2 \circ \rho_1^{-1}:\mathcal{R}&\to\mathfrak{sp}_\phi(P_1,\dots,P_\ell)\\
                    (a_1,\dots,a_\ell)&\mapsto a_1P_1+\cdots+a_\ell P_\ell.
                \end{align*}
        \end{enumerate}
    \end{theorem}

    \begin{proof}
        By \eqref{Eq:Key_Id} and Lemma~\ref{Lem:Ker_1}, we have
        \[
            \mathscr{K}_1=\mathscr{K}_2=\mathcal{R}.
        \]
        Then the first assertion follows immediately from Lemma~\ref{Lem:Ker_2}. The map given in the second assertion is clearly well-defined and surjective. To prove the injectivity, we simply notice that $P_1,\dots,P_\ell$ are linearly independent over $\mathbb{F}_q$ by our assumption, as well as $L$ and $\mathbb{F}_q[t]$ are linearly disjoint over $\mathbb{F}_q$. In particular, if
        \[
            a_1P_1+\cdots+a_\ell P_\ell=0,
        \]
        then $a_1=\cdots=a_\ell=0$. The second assertion now follows. Finally, the third assertion is simply a consequence of the first and the second assertions. This completes the proof.
    \end{proof}

\subsection{An effective LU decomposition for non-commutative operators}
    With Theorem~\ref{Thm:SpPoly} in hand, we have reduced our study of the relation module $\mathcal{R}$ to that of the solution space $\mathrm{Sol}_\phi(P_1,\dots,P_\ell)$. One of the key ingredients of our algorithm is to show that if $(\bm{x},a_1,\dots,a_\ell)\in\mathrm{Sol}_\phi(P_1,\dots,P_\ell)$, then $\bm{x}$ must be in a free $\mathbb{F}_q[t]$-module of finite rank. To achieve this goal, we will verify in the next subsection that $\bm{x}$ is also a solution of some inhomogeneous Frobenius difference equation whose coefficient matrix is upper triangular. For our algorithmic purposes, we need to establish an effective procedure to transfer $\phi_t-t$ to a triangular matrix. Thus, the main result to prove in this section is the following statement which can be seen as an LU decomposition for matrices over the two variable non-commutative ring $L[t,\tau]$.
    
    \begin{theorem}\label{Thm:Triangularization}
        Let $B=(B_{ij})\in\Mat_{d\times d}(L[t,\tau])$ with $\deg_\tau(B)=r$ and $\deg_t(B)=u$. Assume that $L$ is a finite extension over $K$ and $\partial B\in\Mat_{d\times d}(L[t])$ is in Jordan canonical form with $\det(\partial B)\neq 0$. Then for each $2\leq k\leq d$ and any effective divisor $D_k\in\mathrm{Div}(L)$ with
        \begin{equation}\label{Eq:Divisor_Bound}
            \deg(D_k)>\max\{2g_L-2,(rk+1)(k-1)\big(u+(rk+1)(k-1)u+1\big)q^{r(k-1)}\mathrm{ht}(B)_\infty+g_L-1\},
        \end{equation}
        there exists a matrix $A_D\in\Mat_{d\times d}(L[t,\tau])$ of the form
        \begin{equation}\label{Eq:Lower_Triangular_A}
            A_D=\begin{pmatrix}
                1 & & & \\
                a_{21} & a_{22} & & \\
                \vdots & & \ddots & \\
                a_{d1} & \cdots & \cdots & a_{dd}
            \end{pmatrix}
        \end{equation}
        such that $A_DB$ is an upper triangular matrix and the following properties hold.
        \begin{enumerate}
            \item For each $2\leq k\leq d$, $a_{kk}\neq 0$.
            \item The diagonal entries of $A_DB$ are all non-zero.
            \item For each $2\leq k\leq d$, we have
            \[
                \max\{\deg_\tau(a_{k1}),\dots,\deg_\tau(a_{kk})\}\leq r(k-1).
            \] 
            \item For each $2\leq k\leq d$, we have
            \[
                \max\{\deg_t(a_{k1}),\dots,\deg_t(a_{kk})\}\leq (rk+1)(k-1)u.
            \]
            \item For each $2\leq k\leq d$, we have $a_{k\ell}\in\big(\mathcal{L}(D_k)\otimes_{\mathbb{F}_q}\mathbb{F}_q[t]\big)[\tau]$ for any $1\leq\ell\leq k$.
        \end{enumerate}
    \end{theorem}

    We will prove Theorem~\ref{Thm:Triangularization} by using several lemmas. Let $A$ be a matrix of the form \eqref{Eq:Lower_Triangular_A} such that $AB$ is an upper triangular matrix. Then it is equivalent to saying that the $(k,j)$-entry of $AB$ is equal to zero for each $2\leq k\leq d$ and $1\leq j<k$. Thus, for each $2\leq k\leq d$, entries $a_{k\ell}$ of $A$ with $\ell\leq k$ satisfy the following system of equations.
    \begin{equation}\label{Eq:Key_system}
        (\mathbb{E}_k):=\left\{\begin{array}{ccc}
            a_{k1}B_{11}+\cdots+a_{kk}B_{k1}&=&0\\
            a_{k1}B_{12}+\cdots+a_{kk}B_{k2}&=&0\\
            \vdots\\
            a_{k1}B_{1(k-1)}+\cdots+a_{kk}B_{k(k-1)}&=&0
      \end{array}\right.
        \end{equation}
    Here we refer to the system of equations coming from the $k$-th row of $AB$ as $(\mathbb{E}_k)$. Note that for $k_1\neq k_2$, the system of equations $(\mathbb{E}_{k_1})$ and $(\mathbb{E}_{k_2})$ are independent, in the sense that they do not share any common indeterminate $a_{k\ell}$.

    The first observation we have is the following lemma, which shows that any non-zero solution of \eqref{Eq:Key_system} provides the $k$-th row of $A$ with $a_{kk}\neq 0$. This will take care of Theorem~\ref{Thm:Triangularization}(1).

    \begin{lemma}\label{Lem:Non_Zero}
        Let $B=(B_{ij})\in\Mat_{d\times d}(L[t,\tau])$ with $\deg_\tau(B)=r$. Assume that $\partial B\in\Mat_{d\times d}(L[t])$ is in the Jordan canonical form with $\det(\partial B)\neq 0$, and there is a matrix $A\in\Mat_{d\times d}(L[t,\tau])$ of the form \eqref{Eq:Lower_Triangular_A} so that $AB$ is an upper triangular matrix. For any $2\leq k\leq d$, if the $k$-th row of $A$ is not identically zero, then we must have $a_{kk}\neq 0$.
    \end{lemma}

    \begin{proof}
        Suppose on the contrary that there exists $2\leq k\leq d$ such that the $k$-th row of $A$ is not identically zero, but $a_{kk}=0$. Assume that $n_0\geq 0$ is the non-negative integer such that $a_{k\ell}\in L[t,\tau]\tau^{n_0}$ for each $1\leq\ell\leq k$ but $a_{k\ell}\not\in L[t,\tau]\tau^{n_0+1}$ for some $1\leq\ell\leq k$. Then, the $k$-th row of $A$ gives a non-zero solution of the following system of equations        
        \begin{equation}\label{Eq:Key_system'}
            (\mathbb{E}'_k):= \left\{\begin{array}{ccc}
                a_{k1}B_{11}+\cdots+a_{k(k-1)}B_{(k-1)1}&=&0\\
                a_{k1}B_{12}+\cdots+a_{k(k-1)}B_{(k-1)2}&=&0\\
                \vdots\\
                a_{k1}B_{1(k-1)}+\cdots+a_{k(k-1)}B_{(k-1)(k-1)}&=&0
                \end{array}\right.
        \end{equation}
        We express
        \[
            a_{k\ell}=a_{k\ell}^{[n_0]}\tau^{n_0}+\cdots+a_{k\ell}^{[s_k]}\tau^{s_k}\in L[t][\tau],~1\leq\ell\leq k
        \]
        and
        \[
            B_{ij}=B_{ij}^{[0]}+\cdots+B_{ij}^{[r]}\tau^r\in L[t][\tau],~1\leq i,j\leq d.
        \]
        By extracting the $\tau^{n_0}$-terms of \eqref{Eq:Key_system'}, we get
        \begin{equation}\label{Eq:Key_system''}
        \left\{\begin{array}{ccc}
                a_{k1}^{[n_0]}(B_{11}^{[0]})^{q^{n_0}}+\cdots+a_{k(k-1)}^{[n_0]}(B_{(k-1)1}^{[0]})^{q^{n_0}}&=&0\\
                a_{k1}^{[n_0]}(B_{12}^{[0]})^{q^{n_0}}+\cdots+a_{k(k-1)}^{[n_0]}(B_{(k-1)2}^{[0]})^{q^{n_0}}&=&0\\
               \vdots\\
                a_{k1}^{[n_0]}(B_{1(k-1)}^{[0]})^{q^{n_0}}+\cdots+a_{k(k-1)}^{[n_0]}(B_{(k-1)(k-1)}^{[0]})^{q^{n_0}}&=& 0
            \end{array}\right.
        \end{equation}
        It is equivalent to
        \begin{equation}
            (a_{k1}^{[n_0]},\dots,a_{k(k-1)}^{[n_0]})\begin{pmatrix}
                B_{11}^{[0]} & \cdots & B_{1(k-1)}^{[0]}\\
                \vdots & & \vdots\\
                B_{(k-1)1}^{[0]} & \cdots & B_{(k-1)(k-1)}^{[0]}
            \end{pmatrix}^{(n_0)}=(0,\dots,0).
        \end{equation}
        Since $\partial B\in\Mat_{d\times d}(L[t])$ is in the Jordan canonical form with $\det(\partial B)\neq 0$, 
        its principal submatrix of size $(k-1)$ by $(k-1)$ is again in the Jordan canonical form with non-zero determinant. 
        It follows that 
        \[
            \begin{pmatrix}
                B_{11}^{[0]} & \cdots & B_{1(k-1)}^{[0]}\\
                \vdots & & \vdots\\
                B_{(k-1)1}^{[0]} & \cdots & B_{(k-1)(k-1)}^{[0]}
            \end{pmatrix}^{(n_0)}\in\Mat_{(k-1)\times (k-1)}(L[t])\cap\GL_{(k-1)\times (k-1)}(L(t)),
        \]
        and hence $(a_{k1}^{[n_0]},\dots,a_{k(k-1)}^{[n_0]})=(0,\dots,0)$. This implies that $a_{k\ell}\in L[t,\tau]\tau^{n_0+1}$, which violates our assumption $a_{k\ell}\not\in L[t,\tau]\tau^{n_0+1}$. The desired result now follows.
    \end{proof}

    In what follows, we will present some lemmas which can be used to deal with Theorem~\ref{Thm:Triangularization}(2). The conclusions are obvious for matrices over a commutative ring, and the ideas can be adjusted to our non-commutative situation. To be self-contained, we give an elementary explanation without using the notion of determinant. For a square matrix $B\in\Mat_d\big(\oL(t)(\!(\tau)\!)\big)$ and $1\leq i\leq d$, the \emph{leading principal submatrix of $B$ of order $i$} is the upper-left submatrix of $B$ obtained by deleting the last $(d-i)$ rows and columns of $B$.

\begin{lemma}\label{Lem:Non_Zero_Diagonal}
    Let $A,B,C\in\Mat_{d\times d}\big(\oL(t)(\!(\tau)\!)\big)$ be three square matrices of size $d>0$. Assume that
    \begin{enumerate}
        \item $A$ is a lower triangular matrix.
        \item For each $1\leq i\leq d$ and $1\leq j\leq i$, if we denote by $\bv_{i,j}$ the $j$-th row of the leading principal submatrix of $B$ of order $i$, then
        \[
            \bv_{i,i}\not\in\mathrm{Span}_{\oL(t)(\!(\tau)\!)}\{\bv_{i,1},\dots,\bv_{i,i-1}\}.
        \]
        \item $C=AB$ is an upper triangular matrix.
    \end{enumerate}
    Then if the diagonal entries of $A$ are all non-zero, then the diagonal entries of $C$ are all non-zero.
\end{lemma}

\begin{proof}
    Suppose on the contrary that $C_{ii}=0$ for some $1\leq i\leq d$. We claim that this will imply
    \[
        \bv_{i,i}\in\mathrm{Span}_{\oL(t)(\!(\tau)\!)}\{\bv_{i,1},\dots,\bv_{i,i-1}\},
    \]
    which will lead to the desired contradiction.

    To prove the above claim, we first note that since $A_{ii}\neq 0$ for each $1\leq i\leq d$, it follows that $A_{ii}^{-1}\in\oL(t)(\!(\tau)\!)$ exists. Then the matrix
    \[
        N:=\begin{pmatrix}
            A_{11}^{-1} & & \\
             & \ddots & \\
             & & A_{dd}^{-1}
        \end{pmatrix}A-\mathbb{I}_d\in\Mat_d\big(\oL(t)(\!(\tau)\!)\big)
    \]
    is a lower triangular nilpotent matrix. In particular, $N^d$ is the zero matrix. Thus, we have an explicit formula for the inverse of the matrix
    \begin{equation}\label{Eq:Inverse_Formula}
        (\mathbb{I}_d+N)^{-1}=\mathbb{I}_d-N+N^2-N^3+\cdots+(-1)^{d-1}N^{d-1}\in\Mat_d\big(\oL(t)(\!(\tau)\!)\big).
    \end{equation}
    Hence,
    \begin{align}\label{E:BwNAB}
        B&=(\mathbb{I}_d+N)^{-1}\begin{pmatrix}
            A_{11}^{-1} & & \\
             & \ddots & \\
             & & A_{dd}^{-1}
        \end{pmatrix}AB\\
      \notag  &=(\mathbb{I}_d-N+N^2-N^3+\cdots+(-1)^{d-1}N^{d-1})\begin{pmatrix}
            A_{11}^{-1} & & \\
             & \ddots & \\
             & & A_{dd}^{-1}
        \end{pmatrix}C
    \end{align}
    Since $C$ is an upper triangular matrix with $C_{ii}=0$, the leading principal submatrix of
    \[
        \begin{pmatrix}
        A_{11}^{-1} & & \\
         & \ddots & \\
         & & A_{dd}^{-1}
    \end{pmatrix}C
    \]
    order $i$ is of the form 
    \[
        \begin{pmatrix}
        \star & \cdots & \cdots & \star \\
         & \ddots & & \vdots\\
         & & \star & \star\\
         & & & 0
    \end{pmatrix} \in \Mat_i\big(\oL(t)(\!(\tau)\!)\big).
    \]

    On the other hand, since $N$ is a lower triangular matrix, by the explicit formula \eqref{Eq:Inverse_Formula}, the matrix $(\mathbb{I}_d+N)^{-1}$ is also lower triangular. Its leading principal submatrix of order $i$ is of the form 
    \[
        \begin{pmatrix}
            1 & & & \\
            \star & 1 & & \\
            \vdots & & \ddots & \\
            \star & \cdots & \cdots & 1
        \end{pmatrix} \in \Mat_i\big(\oL(t)(\!(\tau)\!)\big).
    \]
    In particular, by \eqref{E:BwNAB}, the leading principal submatrix of $B$ of order $i$ is of the form
    \[
        \begin{pmatrix}
            1 & & & \\
            \star & 1 & & \\
            \vdots & & \ddots & \\
            \star & \cdots & \cdots & 1
        \end{pmatrix}\begin{pmatrix}
        \star & \cdots & \cdots & \star \\
         & \ddots & & \vdots\\
         & & \star & \star\\
         & & & 0
    \end{pmatrix} \in \Mat_i\big(\oL(t)(\!(\tau)\!)\big)
    \]
    which implies that
    \[
        \bv_{i,i}\in\mathrm{Span}_{\oL(t)(\!(\tau)\!)}\{\bv_{i,1},\dots,\bv_{i,i-1}\}.
    \]
    This gives the asserted claim and hence the desired contradiction.
\end{proof}

    The following lemma shows that the second assumption of Lemma~\ref{Lem:Non_Zero_Diagonal} concerning $B$ is fulfilled under the setting of Theorem~\ref{Thm:Triangularization}.

    \begin{lemma}\label{Lem:Jordan_Form}
        Let $B\in\Mat_d(L[t,\tau])$ such that $\partial B$ is in Jordan canonical form and $\det(\partial B)\neq 0$. For each $1\leq i\leq d$ and $1\leq j\leq i$, if we denote by $\bv_{i,j}$ the $j$-th row of the leading principal submatrix of $B$ of order $i$, then
        \[
            \{\bv_{i,1},\dots,\bv_{i,i}\}
        \]
        is an $\oL(t)(\!(\tau)\!)$-linearly independent set. In particular,
        \[
            \bv_{i,i}\not\in\mathrm{Span}_{\oL(t)(\!(\tau)\!)}\{\bv_{i,1},\dots,\bv_{i,i-1}\}.
        \]
    \end{lemma}

    \begin{proof}
        For each $1\leq i\leq d$, since $\partial B$ is in Jordan canonical form and $\det(\partial B)\neq 0$, if we denote by $B_i$ the $i$-th principal submatrix of $B$, then $\partial B_i$ is again in Jordan canonical form and $\det(\partial B_i)\neq 0$.

        Now we assume that there are $a_1,\dots,a_i\in\oL(t)(\!(\tau)\!)$, not all zero, such that
        \begin{equation}\label{Eq:Lttau_relation}
            a_1\bv_{i,1}+\cdots+a_i\bv_{i,i}=0.
        \end{equation}
        Without loss of generality, we may assume that $a_1,\dots,a_i\in\oL(t)\llbracket\tau\rrbracket$ and $\partial(a_1),\dots,\partial(a_i)\in\oL(t)$ are not all zero. Then, by\eqref{Eq:Lttau_relation}, we have 
        \[
            \big(\partial(a_1),\dots,\partial(a_i)\big)\partial B_i=0.
        \]
        Since $\partial B_i$ is invertible, we must have that $\partial(a_1),\dots,\partial(a_i)$ are all zero. This leads to a contradiction.
    \end{proof}

    Given $\mathfrak{B}=(\mathfrak{B}_{ij})\in\Mat_{m\times n}(L)$, we define
    \[
        \mathrm{den}(\mathfrak{B}):=\{v\in M_L\mid\mathrm{ord}_v(\mathfrak{B}_{ij})<0~\mathrm{for~some}~i,j\}\subset M_L.
    \]
    Recall from \S\ref{SS:DivisorsBackground} that $\mathrm{ord}_v(\mathfrak{B})=\min_{i,j}\{\mathrm{ord}_v(\mathfrak{B}_{ij})\}$
    and \eqref{E:divinfty} the divisor of poles
    \[
        D_{\mathfrak{B}}:=\sum_{v\in M_L}-\min\{\mathrm{ord}_v(\mathfrak{B}),0\}\cdot(v)=\sum_{v\in\mathrm{den}(\mathfrak{B})}\big(-\mathrm{ord}_v(\mathfrak{B})\big)\cdot(v)\in\mathrm{Div}(L).
    \]
    Our next lemma will be used to deal with the lower bound of the $v$-adic valuation stated in Theorem~\ref{Thm:Triangularization}(5).

    \begin{lemma}\label{Lem:v_adic_lower_bound}
        Let $\mathfrak{B}=(\mathfrak{B}_{ij})\in\Mat_{m\times n}(L)$ with $m<n$ and $D\in\mathrm{Div}(L)$ be an effective divisor with $\deg(D)>\max\{2g_L-2,\frac{m}{n-m}\deg(D_\mathfrak{B})+g_L-1\}$. Then we have
        \[
            \{0\}\subsetneq\Ker(\mathfrak{B})\cap\mathcal{L}(D)^n.
        \]
    \end{lemma}

\begin{proof}
    We fix $\mathbb{F}_q$-bases for
    \[
        \begin{cases}
            \mathcal{L}(D_\mathfrak{B}),~\mathrm{denoted~by}~\{\alpha_1,\dots,\alpha_{d_{\mathfrak{B}}}\},~\mathrm{where}~d_\mathfrak{B}:=\dim_{\mathbb{F}_q}\mathcal{L}(D_\mathfrak{B}).\\
            \mathcal{L}(D),~\mathrm{denoted~by}~\{\beta_1,\dots,\beta_{d_\fs}\},~\mathrm{where}~d_\fs:=\dim_{\mathbb{F}_q}\mathcal{L}(D).\\
            \mathcal{L}(D_\mathfrak{B}+D),~\mathrm{denoted~by}~\{\gamma_1,\dots,\gamma_{d_\Sigma}\},~\mathrm{where}~d_\Sigma:=\dim_{\mathbb{F}_q}\mathcal{L}(D_\mathfrak{B}+D).
        \end{cases}
    \]
    Then, we express
    \[
        \mathfrak{B}_{ij}=\mathfrak{B}_{ij}^{[1]}\alpha_1+\cdots+\mathfrak{B}_{ij}^{[d_\mathfrak{B}]}\alpha_{d_\mathfrak{B}}, 
    \]
    where each $\mathfrak{B}_{ij}^{[\ell]} \in \FF_q$, 
    and for $\bx=(x_1,\dots,x_n)^\tr\in\mathcal{L}(D)^{n}$ we express
    \[
        x_k=x_k^{[1]}\beta_1+\cdots+x_k^{[d_\fs]}\beta_{d_\fs},
    \]
    where each $x_k^{[\ell]}\in\mathbb{F}_q$. Note that for each $\ell_1, \ell_2$, we can further express
    \[
        \alpha_{\ell_1}\beta_{\ell_2}=\sum_{\ell=1}^{d_\Sigma}h_{\ell_1,\ell_2}^{[\ell]}\gamma_\ell
    \]
    for some $h_{\ell_1,\ell_2}^{[\ell]}\in\mathbb{F}_q$. It follows that
    \[
        \mathfrak{B}\bx=\begin{pmatrix}
            c_1\\
            \vdots\\
            c_m
        \end{pmatrix},
    \]
    where
    \begin{align*}
        c_i=\sum_{j=1}^n\mathfrak{B}_{ij}x_j=\sum_{\ell=1}^{d_\Sigma}\bigg(\sum_{j=1}^n\sum_{\overset{1\leq\ell_1\leq d_\mathfrak{B}}{1\leq\ell_2\leq d_\fs}}\mathfrak{B}_{ij}^{[\ell_1]}x_j^{[\ell_2]}h_{\ell_1,\ell_2}^{[\ell]}\bigg )\gamma_\ell.
    \end{align*}
    If we 
    regard $x_j^{[\ell_2]}$ as indeterminate variables, then when we impose $c_i=0$ for each $1\leq i\leq m$, we obtain a linear system over $\mathbb{F}_q$ with $nd_\fs$ variables and $md_\Sigma$ equations.
    
    Since $D_{\mathfrak{B}}\geq 0$ is an effective divisor, it follows that
    \[
        \deg(D+D_\mathfrak{B})=\deg(D)+\deg(D_{\mathfrak{B}})\geq \deg(D)>2g_L-2.
    \]
    Therefore, by Riemann-Roch theorem, we have 
    \[
        nd_\fs=n(\deg(D)-g_L+1)>m(\deg(D)+\deg(D_{\mathfrak{B}})-g_L+1) =md_\Sigma.
    \]
    Hence, the induced linear system admits non-zero solutions. In particular, the following intersection contains non-zero elements
    \[
        \{0\}\subsetneq\Ker(\mathfrak{B})\cap\mathcal{L}(D)^n.
    \]
\end{proof}

    The proof of the next lemma is in a similar flavor of the proof of Lemma~\ref{Lem:v_adic_lower_bound}. It will lead to the desired $t$-degree bound given in Theorem~\ref{Thm:Triangularization}(4).

\begin{lemma}\label{L:t_degree_bound}
        Let $\mathcal{B}\in\Mat_{m\times n}(L[t])$ with $0<m<n$, $\deg_t(\mathcal{B})=u$.
        Then there exists a non-zero solution $\bx=(x_1,\dots,x_n)^\tr\in\Ker(\mathcal{B})$ with
        \[
            \deg_t(\bx)\leq w:=\left\lfloor\frac{mu}{n-m}\right\rfloor.
        \]
        Furthermore, if we express
        \[
            x_\ell=x_\ell^{[0]}+x_\ell^{[1]}t+\cdots+x_\ell^{[w]}t^{w}\in L[t]
        \]
        for each $1\leq\ell\leq n$, then there exists $\widetilde{\mathcal{B}}\in\Mat_{m(u+w+1)\times n(w+1)}(L)$ with $\ord_v(\widetilde{\mathcal{B}})=\ord_v(\mathcal{B})$ for each $v\in M_L$ such that
        \[
            (x_1^{[0]},\dots,x_1^{[w]},x_2^{[0]},\dots,x_2^{[w]},\dots,x_n^{[0]},\dots,x_n^{[w]})^\tr\in\Ker(\widetilde{\mathcal{B}}).
        \]
    \end{lemma}

    \begin{proof}
        We express $\mathcal{B}=(\mathcal{B}_{ij})$ with
        \[
            \mathcal{B}_{ij}=\mathcal{B}_{ij}^{[0]}+\mathcal{B}_{ij}^{[1]}t+\cdots+\mathcal{B}_{ij}^{[u]}t^{u}.
        \]
       Let $\bx=(x_1,\dots,x_n)^\tr\in\Mat_{n\times 1}(L[t])$ with $\deg_t(\bx)\leq w$. We can write
        \[
            x_\ell=x_\ell^{[0]}+x_\ell^{[1]}t+\cdots+x_\ell^{[w]}t^{w}\in L[t]
        \]
        for each $1\leq\ell\leq n$. If we set $\mathcal{B}\bx=(c_1,\dots,c_m)^\tr\in\Mat_{m\times 1}(L[t])$, then for each $1\leq i\leq m$ we have
        \[
            c_i=\sum_{j=1}^n\mathcal{B}_{ij}x_j=\sum_{j=1}^n\big(\mathcal{B}_{ij}^{[0]}+\cdots+\mathcal{B}_{ij}^{[u]}t^{u}\big)\big(x_j^{[0]}+\cdots+x_j^{[w]}t^{w}\big).
        \]
        It follows that if we express $c_i=c_i^{[0]}+c_i^{[1]}t+\cdots+c_i^{[u+w]}t^{u+w}$, then for $0\leq N\leq u+w$ the coefficient of $t^N$ for $c_i$ is given by 
        \[
            c_i^{[N]}=\sum_{j=1}^n\sum_{\underset{\ell_1+\ell_2=N}{\ell_1,\ell_2\in\mathbb{Z}_{\geq 0}}}\mathcal{B}_{ij}^{[\ell_1]}x_j^{[\ell_2]}.
        \]
        Now we treat $x_j^{[\ell_2]}$ as indeterminate. Then we have $n(w+1)$ variables. If we require $\bx\in\Ker(\mathcal{B})$, then we must have $c_i^{[N]}=0$,  which induces at most $m(u+w+1)$ linear equations over $L$. Using our assumption that $w=\left\lfloor\frac{mu}{n-m}\right\rfloor>\frac{mu}{n-m}-1$, we must have
        \[
            n(w+1)>m(u+w+1).
        \]
        The existence of $0\neq\bx\in\Ker(\mathcal{B})$ with $\deg_t(\bx)\leq w$ now follows. Note that the induced linear system over $L$ from
        \[
            \sum_{j=1}^n\sum_{\underset{\ell_1+\ell_2=N}{\ell_1,\ell_2\in\mathbb{Z}_{\geq 0}}}\mathcal{B}_{ij}^{[\ell_1]}x_j^{[\ell_2]}=0
        \]
        can be described by a matrix $\widetilde{\mathcal{B}}\in\Mat_{m(u+w+1)\times n(w+1)}(L)$ with $\ord_v(\widetilde{\mathcal{B}})=\ord_v(\mathcal{B})$ for each $v\in M_L$ such that
        \[
            (x_1^{[0]},\dots,x_1^{[w]},x_2^{[0]},\dots,x_2^{[w]},\dots,x_n^{[0]},\dots,x_n^{[w]})^\tr\in\Ker(\widetilde{\mathcal{B}}).
        \]
        This completes the proof. 
    \end{proof}

    \begin{proof}[Proof of Theorem~\ref{Thm:Triangularization}]
        If we construct the matrix $A_D$ whose $k$-th row comes from the solution of $(\mathbb{E}_k)$ given in $\eqref{Eq:Key_system}$, then $A_DB$ must be an upper triangular matrix. Moreover, if we have a non-trivial solution of $(\mathbb{E}_k)$, then Lemma~\ref{Lem:Non_Zero} implies that $a_{kk}\neq 0$. Also, it follows from Lemma~\ref{Lem:Non_Zero_Diagonal} and Lemma~\ref{Lem:Jordan_Form} that (2) holds. Thus, our task reduces to constructing non-zero solutions for the system of equations $(\mathbb{E}_k)$ such that they satisfy the desired properties (3), (4), and (5).

        Now we fix $k$ for some $2\leq k\leq d$, and focus on the system of equations $(\mathbb{E}_k)$.
        We express
        \[
            B_{ij}=B_{ij}^{[0]}+\cdots+B_{ij}^{[r]}\tau^r\in L[t][\tau],
        \]
        and
        \[
            a_{k\ell}=a_{k\ell}^{[0]}+\cdots+a_{k\ell}^{[s_k]}\tau^{s_k}\in L[t][\tau],
        \]
        where $s_k$ is a fixed integer with $s_k\geq r(k-1)$. 
        It follows that
        \[
            \sum_{i=1}^ka_{ki}B_{ij}=\sum_{m=0}^{r+s_k}c_{kj}^{[m]}\tau^m,
        \]
        where
        \begin{equation}\label{E:linsystaudeg}
            c_{kj}^{[m]}=\sum_{i=1}^k\sum_{\overset{n_1,n_2\geq 0}{n_1+n_2=m}}a_{ki}^{[n_1]}(B_{ij}^{[n_2]})^{(n_1)}\in L[t].
        \end{equation}
        Here we adopt the convention that $a_{ki}^{[n_1]}=0$ if $n_1>s_k$ and $B_{ij}^{[n_2]}=0$ if $n_2>r$.

        By imposing the coefficient of $\tau^m$, $0 \leq m \leq r+s_k$, equals zero in each of the $k-1$ equations in $(\mathbb{E}_k)$ in \eqref{Eq:Key_system}, the number of the induced equations over $L[t]$ is at most $(r+s_k+1)(k-1)$. Now we treat $a_{ki}^{[m]}$ as indeterminates. Then, the number of indeterminates $a_{ki}^{[m]}$ is given by $(s_k+1)k$. On the other hand, since $s_k\geq r(k-1)$, we must have
        \[
            (s_k+1)k>(r+s_k+1)(k-1).
        \]
        It follows that the system of equations $(\mathbb{E}_k)$ given in \eqref{Eq:Key_system} admits non-zero solutions. Note that the induced system of equations over $L[t]$ can be described explicitly by a matrix 
        \[
            \widetilde{B}_k\in\Mat_{(r+s_k+1)(k-1)\times(s_k+1)k}(L[t]).
        \]
        Here we allow some of the rows of $\widetilde{B}_k$ to be zero if necessary. By our construction, we must have $\deg_t(\widetilde{B}_k)\leq\deg_t(B)$ and  $\ord_v(\widetilde{B}_k)\geq\min\{0,q^{s_k}\ord_v(B)\}$ for each $v\in M_L$ such that
        \[
            (a_{k1}^{[0]},\dots,a_{k1}^{[s_k]},a_{k2}^{[0]},\dots,a_{k2}^{[s_k]},\dots,a_{kk}^{[0]},\dots,a_{kk}^{[s_k]})^\tr\in\Ker(\widetilde{B}_k).
        \]

        Now we fix $s_k=r(k-1)$ and specialize $\mathcal{B}=\widetilde{B}_k$ in Lemma~\ref{L:t_degree_bound}. In this case,
        \begin{align*}
            (s_k+1)k-(r+s_k+1)(k-1)&=s_kk+k-rk+r-s_kk+s_k-k+1\\
            &=-rk+r+s_k+1\\
            &=-rk+r+r(k-1)+1\\
            &=1,
        \end{align*}
        and thus there is a non-zero element in $\Ker(\widetilde{B}_k)$ whose $t$-degree is bounded above by
        \[
            \left\lfloor\frac{(r+s_k+1)(k-1)u}{(s_k+1)k-(r+s_k+1)(k-1)}\right\rfloor=(rk+1)(k-1)u.
        \]
        Furthermore, we can specialize the matrix $\mathfrak{B}$ in Lemma~\ref{Lem:v_adic_lower_bound} to the matrix $\widetilde{\mathcal{B}}$ constructed in Lemma~\ref{L:t_degree_bound}. In this case, the size of $\mathfrak{B}=\Tilde{\mathcal{B}}$ is $(rk+1)(k-1)\big(u+(rk+1)(k-1)u+1\big)$ by $(rk-r+1)k\big((rk+1)(k-1)u+1\big)$ and $\mathrm{ht}(\mathfrak{B})_\infty\leq q^{s_k}\mathrm{ht}(B)_\infty=q^{r(k-1)}\mathrm{ht}(B)_\infty$. It follows from Lemma~\ref{Lem:v_adic_lower_bound} that for each $2\leq k\leq d$ we can have non-trivial solutions with
        \[
            a_{ki}\in\big(\mathcal{L}(D_k)\otimes_{\mathbb{F}_q}\mathbb{F}_q[t]\big)[\tau]
        \]
        for every $1\leq i\leq d$. The desired result now follows.
    \end{proof}

    \begin{example}\label{Ex:3rdCTs}
        Let $\mathbf{C}^{\otimes 3}=(\mathbb{G}_a^3,[\cdot]_3)$ be the $3$-rd tensor powers of the Carlitz module defined by
        \begin{align*}
            [\cdot]_3:\mathbb{F}_q[t]&\to\Mat_{3}(K[\tau])\\
            a&\mapsto[a]_3
        \end{align*}
        which is uniquely determined by $[t]_3=\begin{pmatrix}
            \theta & 1 & \\
             & \theta & 1\\
             \tau & & \theta
        \end{pmatrix}$. 
    Let \[B=[t]_3-t\Id_3 =  \begin{pmatrix}
                \theta-t & 1 & \\
                 & \theta-t & 1\\
                 \tau & & \theta-t
            \end{pmatrix}.\]
            Using the method presented in this subsection, we observe that for $k=2$, we want to solve for $a_{21},a_{22}\in K[t,\tau]$ such that
        \[
            a_{21}(\theta-t)=0.
        \]
        It follows that we can choose $a_{21}=0$ and $a_{22}=1$. 
        
        For $k=3$, we are seeking for $a_{31},a_{32},a_{33}\in K[t,\tau]$ such that
        \begin{equation}\label{Eq:3TC_1}
            \begin{cases}
                a_{31}(\theta-t)+a_{33}\tau=0\\
                a_{31}+a_{32}(\theta-t)=0
            \end{cases}.
        \end{equation}
        One may use Theorem~\ref{Thm:Triangularization}(3) to see that there is a non-trivial choice with
        \[
            \max\{\deg_\tau(a_{31}),\deg_\tau(a_{32}),\deg_\tau(a_{33})\}\leq r(k-1)=1\cdot(3-1)=2,
        \]
        and then set up a system of linear equations over $K[t]$ with $4$ linear relations and $6$ variables. The complete set of solutions will produce all the possible $a_{31},a_{32},a_{33}$. Or one can perform a direct substitution: the second line of \eqref{Eq:3TC_1} is equivalent to $a_{31}=-a_{32}(\theta-t)$. Then the first equation can be rewritten as $a_{33}\tau=a_{32}(\theta-t)^2$. It follows that $a_{32}=\tau$, $a_{33}=(t-\theta^q)^2$, and $a_{31}=-a_{32}(\theta-t)=(t-\theta^q)\tau$ are valid solution.

        Finally, one can check
        \[
            \begin{pmatrix}
                1 & & \\
                 & 1 & \\
                (t-\theta^q)\tau & \tau & (t-\theta^{q})^2
            \end{pmatrix}\begin{pmatrix}
            \theta-t & 1 & \\
             & \theta-t & 1\\
             \tau & & \theta-t
        \end{pmatrix}=\begin{pmatrix}
            \theta-t & 1 & \\
             & \theta-t & 1\\
             & & \tau-(t-\theta)(t-\theta^q)^2 
        \end{pmatrix}
        \]
        is an upper triangular matrix.
    \end{example}

\subsection{Solutions of inhomogeneous Frobenius difference equations}
    In what follows, we aim to investigate the solution space of some specific types of Frobenius difference equations. Let $d>0$ be a fixed positive integer, and $Q_1,\dots,Q_\ell\in\Mat_{d\times 1}(L)$ be some fixed column vectors with entries in $L$. Given $\mathcal{A}\in\Mat_{d\times d}(L[t,\tau])$, we set
    \[
        \mathrm{Sol}(\mathcal{A};Q_1,\dots,Q_\ell):=\{(\bm{x},b_1,\dots,b_\ell)\in\Mat_{d\times 1}(L[t])\times\mathbb{F}_q[t]^{\oplus\ell}\mid \mathcal{A}\bm{x}=b_1Q_1+\cdots+b_\ell Q_\ell\}.
    \]
    Furthermore, we define
    \begin{align*}
        \mathrm{pr}_1:\Mat_{d\times 1}(L[t])\times\mathbb{F}_q[t]^{\oplus\ell}&\to\Mat_{d\times 1}(L[t])\\
        (\bm{x},b_1,\dots,b_\ell)&\mapsto\bm{x}
    \end{align*}

    The main result to obtain in this section is the following statement. Since the $\mathbb{F}_q$-basis of the Riemann-Roch space associated to a divisor can be computed effectively, we then obtain an explicit way to regard $\mathrm{Sol}_\phi(P_1,\dots,P_\ell)$ as an $\mathbb{F}_q[t]$-submodule of a free $\mathbb{F}_q[t]$-module of finite rank. This will allow us to prove in the next subsection that $\mathrm{Sol}_\phi(P_1,\dots,P_\ell)$ can be computed by a system of linear equations over $\mathbb{F}_q[t]$.

    \begin{theorem}\label{Thm:Sol_Divisor}
        Let $E=(\mathbb{G}_a^d,\phi)$ be a $d$-dimensional $t$-module defined over $L$ and $P_1,\dots,P_\ell\in E(L)=\Mat_{d\times 1}(L)$. 
        Then there is an explicitly constructed effective divisor $D_{\mathcal{P}}\in\mathrm{Div}(L)$ such that
        \[
            \mathrm{pr}_1\big(\mathrm{Sol}_\phi(P_1,\dots,P_\ell)\big)\subset\Mat_{d\times 1}(\mathcal{L}(D_{\mathcal{P}})\otimes_{\mathbb{F}_q}\mathbb{F}_q[t]).
        \]
    \end{theorem}

    We postpone the proof until we have established essential preliminary results.
    The following observation is the starting point of the reduction step.
    \begin{lemma}\label{Lem:Emb}
        Let $E=(\mathbb{G}_a^d,\phi)$ be a $d$-dimensional $t$-module defined over $L$, $P_1,\dots,P_\ell\in E(L)=\Mat_{d\times 1}(L)$. 
        Let $A\in\Mat_{d\times d}(L[t,\tau])$, 
        and set $\mathcal{A}:=A(\phi_t-t)$ and for each $1\leq i \leq \ell$, $Q_i:=AP_i\in\Mat_{d\times 1}(L[t])$. Then the natural inclusion
        \[
            \mathrm{Sol}_\phi(P_1,\dots,P_\ell)\hookrightarrow\mathrm{Sol}(\mathcal{A};Q_1,\dots,Q_\ell)
        \]
        is well-defined. 
    \end{lemma}

    \begin{proof}
        Given $(\bm{x},a_1,\dots,a_\ell)\in\mathrm{Sol}_\phi(P_1,\dots,P_\ell)$, we have
        \[
            (\phi_t-t)\bm{x}=a_1P_1+\cdots+a_\ell P_\ell.
        \]
        It follows that
        \[
            \mathcal{A}\bm{x}=A(\phi_t-t)\bm{x}=A(a_1P_1+\cdots+a_\ell P_\ell)=a_1Q_1+\cdots+a_\ell Q_\ell
        \]
        which verifies that the inclusion is indeed well-defined.
    \end{proof}

    In order to prove Theorem~\ref{Thm:Sol_Divisor}, we need some estimations on the $v$-adic valuation.

    \begin{lemma}\label{Lem:1D_Case_New}
        Let $D=\sum_{v\in M_L}a_v\cdot(v)\in\mathrm{Div}(L)$ be an effective divisor. Consider $\mathcal{A}=\kappa_0+\kappa_1\tau+\cdots+\kappa_r\tau^r\in L[t,\tau]$ with $\kappa_0,\dots,\kappa_r\in L[t]$, and $\bm{y}\in\mathcal{L}(D)\otimes_{\mathbb{F}_q}\mathbb{F}_q[t]\subset L[t]$. If $\bm{x}\in L[t]$ satisfies $\mathcal{A}\bm{x}=\bm{y}$, then $\bm{x}\in\mathcal{L}(D_\mathcal{A})\otimes_{\mathbb{F}_q}\mathbb{F}_q[t]$ where $D_{\mathcal{A}}$ is the effective divisor defined by
        \begin{equation}\label{E:divA}
            D_{\mathcal{A}}:=\sum_{v\in M_L}\max_{0\leq j\leq r-1}\left\{-\left\lceil\frac{\ord_v(\kappa_j)-\ord_v(\kappa_r)}{q^r-q^j}\right\rceil,-\left\lceil\frac{-a_v-\ord_v(\kappa_r)}{q^r}\right\rceil,0\right\}\cdot(v)\in\mathrm{Div}(L).
        \end{equation}
    \end{lemma}

    \begin{proof}
        It suffices to show that if $\bm{x}\in L[t]$ with
        \begin{equation}\label{Eq:DifferenceEq}
            \kappa_r\bm{x}^{(r)}+\cdots+\kappa_1\bm{x}^{(1)}+\kappa_0\bm{x}=\bm{y},
        \end{equation}
        then for $v\in M_L$ we have
        \[
            \mathrm{ord}_v(\bm{x})\geq\min_{0\leq j\leq r-1}\left\{\frac{\mathrm{ord}_v(\kappa_j)-\mathrm{ord}_v(\kappa_r)}{q^r-q^j},\frac{-a_v-\mathrm{ord}_v(\kappa_r)}{q^r}\right\}.
        \]
        Suppose on the contrary that
        \[
            \mathrm{ord}_v(\bm{x})<\min_{0\leq j\leq r-1}\left\{\frac{\mathrm{ord}_v(\kappa_j)-\mathrm{ord}_v(\kappa_r)}{q^r-q^j},\frac{-a_v-\mathrm{ord}_v(\kappa_r)}{q^r}\right\}.
        \]
        It follows that for each $0\leq j\leq r-1$, we have
        \[
            \mathrm{ord}_v(\bm{x})<\frac{\mathrm{ord}_v(\kappa_j)-\mathrm{ord}_v(\kappa_r)}{q^r-q^j},
        \]
        or equivalently
        \[
            \mathrm{ord}_v(\kappa_r\bm{x}^{(r)})<\mathrm{ord}_v(\kappa_j\bm{x}^{(j)}).
        \]
        This implies that
        \[
            \mathrm{ord}_v\big(\kappa_r\bm{x}^{(r)}+\cdots+\kappa_1\bm{x}^{(1)}+\kappa_0\bm{x}\big)=\mathrm{ord}_v(\kappa_r\bm{x}^{(r)})=\mathrm{ord}_v(\kappa_r)+q^r\mathrm{ord}_v(\bm{x}).
        \]
        By using \eqref{Eq:DifferenceEq}, we must have
        \[
            \mathrm{ord}_v(\kappa_r)+q^r\mathrm{ord}_v(\bm{x})=\mathrm{ord}_v(\bm{y})\geq -a_v
        \]
        which contradicts to the assumption that 
        \[
            \mathrm{ord}_v(\bm{x})<\frac{-a_v-\mathrm{ord}_v(\kappa_r)}{q^r}.
        \]
        The desired inequality now follows.
    \end{proof}
    
    \begin{lemma}\label{Lem:1D_Case_New_2}
        Let $D=\sum_{v\in M_L}a_v\cdot(v)\in\mathrm{Div}(L)$ be an effective divisor. Consider $\mathcal{A}=\kappa_0+\kappa_1\tau+\cdots+\kappa_r\tau^r\in L[t,\tau]$ with $\kappa_0,\dots,\kappa_r\in L[t]$, and $\bm{x}\in\mathcal{L}(D)\otimes_{\mathbb{F}_q}\mathbb{F}_q[t]\subset L[t]$. Then $\bm{y}:=\mathcal{A}\bm{x}\in\mathcal{L}\big(\mathcal{A}(D)\big)\otimes_{\mathbb{F}_q}\mathbb{F}_q[t]$ where $\mathcal{A}(D)$ is the effective divisor defined by
        \begin{equation}\label{Eq:divAD}
            \mathcal{A}(D):=\sum_{v\in M_L}\max_{0\leq j\leq r}\{q^ja_v-\ord_v(\kappa_j),0\}\cdot(v)\in\mathrm{Div}(L).
        \end{equation}
    \end{lemma}

    \begin{proof}
        Since $\bm{x}\in\mathcal{L}(D)\otimes_{\mathbb{F}_q}\mathbb{F}_q[t]$, for each $v\in M_L$ we have
        \[
            \ord_v(\bm{x})\geq -a_v.
        \]
        It follows that
        \begin{align*}
            \ord_v(\bm{y})&=\ord_v(\mathcal{A}\bm{x})\\
            &=\ord_v(\kappa_r\bm{x}^{(r)}+\cdots+\kappa_1\bm{x}^{(1)}+\kappa_0\bm{x})\\
            &\geq\min_{0\leq j\leq r}\{\ord_v(\kappa_j)+q^j\ord_v(\bm{x})\}\\
            &\geq\min_{0\leq j\leq r}\{\ord_v(\kappa_j)-q^ja_v\}.
        \end{align*}
        This leads to the desired result.
    \end{proof}
    
    \begin{lemma}\label{Lem:Inductive_Bounds_New}
        Let $D_i\in\mathrm{Div}(L)$ be effective divisors for each $1\leq i\leq d$. Let $\mathcal{A}=(\mathcal{A}_{ij})\in\Mat_{d}(L[t,\tau])$ be an upper triangular matrix with $\mathcal{A}_{ii}\neq 0$ and $\bm{y}=(y_1,\dots,y_d)^{\tr}\in\Mat_{d\times 1}(L[t])$ with $y_i\in\mathcal{L}(D_i)\otimes_{\mathbb{F}_q}\mathbb{F}_q[t]$. If $\bm{x}=(x_1,\dots,x_d)^\tr\in\Mat_{d\times 1}(L[t])$ satisfies $\mathcal{A}\bm{x}=\bm{y}$, then for each $1\leq i\leq d$ we have $x_i\in\mathcal{L}(\mathscr{D}_i)\otimes_{\mathbb{F}_q}\mathbb{F}_q[t]$, where $\mathscr{D}_i\in\mathrm{Div}(L)$ are effective divisors recursively defined by
        \[
            \mathscr{D}_d:=(D_d)_{\mathcal{A}_{dd}}
        \]
        as in \eqref{E:divA}, and for $1\leq i\leq d-1$,
        \[
            \mathscr{D}_i:=\big(D_i\vee\mathcal{A}_{i(i+1)}(\mathscr{D}_{i+1})\vee\cdots\vee\mathcal{A}_{id}(\mathscr{D}_{d})\big)_{\mathcal{A}_{ii}}\in\mathrm{Div}(L)
        \]
        as in \eqref{E:divA} and $\eqref{Eq:divAD}$.
    \end{lemma}

    \begin{proof}
        For the case of $i=d$, since $\mathcal{A}$ is upper triangular, the last row of the relation $\mathcal{A}\bm{x}=\bm{y}$ is simply $\mathcal{A}_{dd}x_d=y_d$. Thus, Lemma~\ref{Lem:1D_Case_New} shows that $x_d\in\mathcal{L}(\mathscr{D}_d)\otimes_{\mathbb{F}_q}\mathbb{F}_q[t]$. Now we fix $i$ and assume that for each $i< j\leq d$ we have shown that $x_j\in\mathcal{L}(\mathscr{D}_j)\otimes_{\mathbb{F}_q}\mathbb{F}_q[t]$. The $i$-th row of the difference equation $\mathcal{A}\bm{x}=\bm{y}$ provides the difference equation
        \[
            \mathcal{A}_{ii}x_i+\mathcal{A}_{i(i+1)}x_{i+1}+\cdots+\mathcal{A}_{ir}x_r=y_i,
        \]
        or equivalently
        \[
            \mathcal{A}_{ii}x_i=y_i-(\mathcal{A}_{i(i+1)}x_{i+1}+\cdots+\mathcal{A}_{ir}x_r).
        \]
        Then Lemma~\ref{Lem:1D_Case_New}, Lemma~\ref{Lem:1D_Case_New_2}, and the hypothesis imply that $x_i\in\mathcal{L}(\mathscr{D}_i)\otimes_{\mathbb{F}_q}\mathbb{F}_q[t]$, which leads to the desired assertion.
    \end{proof}

    We can now prove Theorem~\ref{Thm:Sol_Divisor}.
    
    \begin{proof}[Proof of Theorem~\ref{Thm:Sol_Divisor}]
        We first prove the result when $\partial\phi_t$ is in Jordan canonical form. By applying Theorem~\ref{Thm:Triangularization} to the case of $B=\phi_t-t$ with any divisor satisfying \eqref{Eq:Divisor_Bound}. Then we get an explicitly constructed lower triangular matrix $A_D$ of the form \eqref{Eq:Lower_Triangular_A} such that $\mathcal{A}:=A_D(\phi_t-t)$ is upper triangular. For $1\leq j\leq \ell$, let $Q_j:=A_DP_j=(Q_j^{[1]},\dots,Q_j^{[d]})^\tr\in\Mat_{d\times 1}(L)$. For $1\leq i\leq \ell$, consider the divisor of poles $D_i:=\mathrm{div}(Q_1^{[i]},\dots,Q_\ell^{[i]})_\infty\in\mathrm{Div}(L)$, as defined in \eqref{E:divinfty} for the matrix $(Q_1^{[i]},\dots,Q_{\ell}^{[i]})\in\Mat_{1\times \ell}(L)$. Then for any $\ell$-tuple $(a_1,\dots,a_\ell)\in\mathbb{F}_q[t]^{\oplus\ell}$, if we define
        \[
            \bm{y}:=\bm{y}(a_1,\dots,a_\ell):=a_1Q_1+\cdots+a_\ell Q_\ell\in\Mat_{d\times 1}(L[t])
        \]
        and write $\bm{y}=(y_1,\dots,y_d)^\tr$, then we must have $y_i\in\mathcal{L}(D_i)\otimes_{\mathbb{F}_q}\mathbb{F}_q[t]$. It follows from Lemma~\ref{Lem:Inductive_Bounds_New} that there are explicitly constructed effective divisors $\mathscr{D}_i$ such that
        \[
            \mathrm{pr}_1\big(\mathrm{Sol}(\mathcal{A};Q_1,\dots,Q_\ell)\big)\subset\Mat_{d\times 1}(\mathcal{L}(D_{\mathcal{P}})\otimes_{\mathbb{F}_q}\mathbb{F}_q[t]),
        \]
        where $D_{\mathcal{P}}:=\mathscr{D}_1\vee\cdots\vee\mathscr{D}_d$.
        Then the desired result now follows from Lemma~\ref{Lem:Emb} that
        \[
            \mathrm{Sol}_\phi(P_1,\dots,P_\ell)\hookrightarrow\mathrm{Sol}(\mathcal{A};Q_1,\dots,Q_\ell).
        \]

        Now we deal with the case of arbitrary $t$-module $E=(\mathbb{G}_a^d,\phi)$. Since $\partial\phi_t=\theta\Id_d+N$ for some nilpotent matrix $N\in\Mat_{d\times d}(L)$, there is $\mathrm{J} \in \GL_d(L)$ such that $\mathrm{J}(\partial\phi_t)\mathrm{J}^{-1}$ is in Jordan canonical form. Consider the $t$-module $\Tilde{E}=(\mathbb{G}_a^d,\Tilde{\phi})$ defined by $\Tilde{\phi}_t:=\mathrm{J}\phi_t\mathrm{J}^{-1}$, and points $\Tilde{P}_1:=\mathrm{J}P_1,\dots,\Tilde{P}_\ell=\mathrm{J}P_\ell$. Then by the result we just proved above, there is an explicitly constructed effective divisor $D_{\Tilde{\mathcal{P}}}$ such that
        \[
            \mathrm{pr}_1\big(\mathrm{Sol}_{\Tilde{\phi}}(\Tilde{P}_1,\dots,\Tilde{P}_\ell)\big)\subset\Mat_{d\times 1}(\mathcal{L}(D_{\tilde{\mathcal{P}}})\otimes_{\mathbb{F}_q}\mathbb{F}_q[t]).
        \]
        We claim that the following map
        \begin{align*}
            \mathrm{Sol}_{\Tilde{\phi}}(\Tilde{P}_1,\dots,\Tilde{P}_\ell)&\to\mathrm{Sol}_{\phi}(P_1,\dots,P_\ell)\\
            (\Tilde{\bm{x}},a_1,\dots,a_\ell)&\mapsto(\mathrm{J}^{-1}\tilde{\bm{x}},a_1,\dots,a_\ell)
        \end{align*}
        is a bijection. Indeed,
        \[
            (\Tilde{\phi}_t-t)\Tilde{\bm{x}}=a_1\Tilde{P}_1+\cdots+a_\ell\Tilde{P}_\ell
        \]
        is equivalent to
        \[
            \mathrm{J}(\phi_t-t)\mathrm{J}^{-1}\Tilde{\bm{x}}=a_1\mathrm{J}P_1+\cdots+a_\ell \mathrm{J}P_\ell.
        \]
        It holds if and only if
        \[
            (\phi_t-t)(\mathrm{J}^{-1}\Tilde{\bm{x}})=a_1P_1+\cdots+a_\ell P_\ell,
        \]
        which leads to our claim. It follows that if we set
        \[
            D_{\mathcal{P}}
            :=\bigvee_{i,j}(\mathrm{J}^{-1})_{ij}(D_{\Tilde{\mathcal{P}}})
        \]
        with $(\mathrm{J}^{-1})_{ij}(D_{\Tilde{\mathcal{P}}})$ defined as in \eqref{Eq:divAD}, then by the claim just verified, we get
         \[
            \mathrm{pr}_1\big(\mathrm{Sol}_\phi(P_1,\dots,P_\ell)\big)\subset\Mat_{d\times 1}(\mathcal{L}(D_{\mathcal{P}})\otimes_{\mathbb{F}_q}\mathbb{F}_q[t]).
        \]
    \end{proof}

    \begin{example}\label{Ex:3rdCTs_Pts}
        We continue with Example~\ref{Ex:3rdCTs}. Let $P_1=(0,0,1)^\tr$ and $P_2=(0,0,\theta)^\tr$. Then we have
        \[
            Q_1=\begin{pmatrix}
                1 & & \\
                 & 1 & \\
                (t-\theta^q)\tau & \tau & (t-\theta^{q})^2
            \end{pmatrix}\begin{pmatrix}
                0\\
                0\\
                1
            \end{pmatrix}=\begin{pmatrix}
                0\\
                0\\
                (t-\theta^q)^2
            \end{pmatrix}
        \]
        and
        \[
            Q_2=\begin{pmatrix}
                1 & & \\
                 & 1 & \\
                (t-\theta^q)\tau & \tau & (t-\theta^{q})^2
            \end{pmatrix}\begin{pmatrix}
                0\\
                0\\
                \theta
            \end{pmatrix}=\begin{pmatrix}
                0\\
                0\\
                (t-\theta^q)^2\theta
            \end{pmatrix}.
        \]
        Consider
        \[
            \mathcal{A}=\begin{pmatrix}
                \theta-t & 1 & \\
                 & \theta-t & 1\\
                 & & \tau-(t-\theta)(t-\theta^q)^2 
            \end{pmatrix}.
        \]
        Following Lemma~\ref{Lem:1D_Case_New}, for $a_1,a_2\in\mathbb{F}_q[t]$, we see that 
        \[
            a_1(t-\theta^q)^2+a_2(t-\theta^q)^2\theta\in\mathcal{L}\big((2q+1)\cdot(\infty)\big).
        \]
        It follows that if $x_3\in K[t]$ satisfies
        \[
            \big(\tau-(t-\theta)(t-\theta^q)^2\big)(x_3)=a_1(t-\theta^q)^2+a_2(t-\theta^q)^2\theta
        \]
        then we must have $\mathrm{ord}_v(x_3)\geq 0$ for $\infty\neq v\in M_K$ and
        \[
            \mathrm{ord}_\infty(x_3)\geq\min\left\{-\frac{2q+1}{q-1},-\frac{2q+1}{q}\right\}=-2-\frac{3}{q-1}.
        \]
        Hence $\mathscr{D}_3=\left(2+\left\lfloor\frac{3}{q-1}\right\rfloor\right)\cdot(\infty)$ and $x_3\in\mathcal{L}(\mathscr{D}_3)\otimes_{\mathbb{F}_q}\mathbb{F}_q[t]$.
        Inductively, following Lemma~\ref{Lem:Inductive_Bounds_New}, we have $\mathscr{D}_2=\left(1+\left\lfloor\frac{3}{q-1}\right\rfloor\right)\cdot(\infty)$ and $\mathscr{D}_1=\left\lfloor\frac{3}{q-1}\right\rfloor\cdot(\infty)$. In particular, for $x_1,x_2\in K[t]$ with $(\theta-t)x_2+x_3=0$ and $(\theta-t)x_1+x_2=0$, we have
        \[
            x_2\in\mathcal{L}(\mathscr{D}_2)\otimes_{\mathbb{F}_q}\mathbb{F}_q[t]
        \]
        and
        \[
            x_1\in\mathcal{L}(\mathscr{D}_1)\otimes_{\mathbb{F}_q}\mathbb{F}_q[t]
        \]
        For the convenience of later use, we remark that this means $x_1,x_2,x_3\in A[t]$ with
        \[
            \begin{cases}
                \deg_\theta(x_1)\leq\frac{3}{q-1}\\
                \deg_\theta(x_2)\leq 1+\frac{3}{q-1}\\
                \deg_\theta(x_3)\leq 2+\frac{3}{q-1}
            \end{cases}.
        \]
    \end{example}

\subsection{The induced system of linear equations over \texorpdfstring{$\mathbb{F}_q[t]$}{}}
    The main purpose in this subsection is to demonstrate how we determine $\mathrm{Sol}_\phi(P_1,\dots,P_\ell)$ using an explicitly constructed system of linear equations over $\mathbb{F}_q[t]$. Let $V\subset L$ be a finite dimensional $\mathbb{F}_q$-vector space with $n:=\dim_{\mathbb{F}_q}V$. We fix an $\mathbb{F}_q$-basis $\{\gamma_1,\dots,\gamma_n\}$ for $V$.  
    For each $f\in V\otimes_{\mathbb{F}_q}\mathbb{F}_q[t]\subset L[t]$, there are unique $f_1,\dots,f_n\in\mathbb{F}_q[t]$ such that $f=f_1\gamma_1+\cdots+f_n\gamma_n$.
    Indeed, let $f=a_0+\cdots+a_st^s\in L[t]$ with $a_0,\dots,a_s\in V$. Then each of the coefficient can be written uniquely as an $\mathbb{F}_q$-linear combination of $\gamma_1,\dots,\gamma_n$, that is, for each $0\leq i\leq s$ we have
    \[
        a_i=a_i^{[1]}\gamma_1+\cdots+a_i^{[n]}\gamma_n,~a_i^{[j]}\in\mathbb{F}_q.
    \]
    Thus, if we set $f_j:=\sum_{i=0}^sa_i^{[j]}t^i$, then we have the unique expression
    \[
        f=\sum_{i=0}^s(\sum_{j=1}^na_i^{[j]}\gamma_j)t^i=\sum_{j=1}^n(\sum_{i=0}^sa_i^{[j]}t^i)\gamma_j=\sum_{j=1}^nf_j\gamma_j.
    \]
     It induces an $\mathbb{F}_q[t]$-linear bijection
    \begin{align*}
        \iota_V:V\otimes_{\mathbb{F}_q}\mathbb{F}_q[t]&\to\Mat_{1\times n}(\mathbb{F}_q[t])\\
        f&\mapsto(f_1,\dots,f_n).
    \end{align*}
    For $\bm{x}=(x_1,\dots,x_d)^\tr\in\Mat_{d\times 1}(V\otimes_{\mathbb{F}_q}\mathbb{F}_q[t])$, we define
    \[
        \iota_V(\bm{x}):=(\iota_V(x_1),\dots,\iota_V(x_d))\in\Mat_{1\times nd}(\mathbb{F}_q[t]).
    \]

    Now we consider two finite dimensional $\mathbb{F}_q$-vector spaces $V,W\subset L$ with $n:=\dim_{\mathbb{F}_q}V$ and $m:=\dim_{\mathbb{F}_q}W$. Let $\Phi\in L[t,\tau]$ with $\Phi\big(V\otimes_{\mathbb{F}_q}\mathbb{F}_q[t]\big)\subset W\otimes_{\mathbb{F}_q}\mathbb{F}_q[t]$. Since $\Phi$ can be treated as an $\mathbb{F}_q[t]$-linear endomorphism on $L[t]$, the relation 
    \[
        \iota_{W}\bigg(\Phi\big(\iota^{-1}_{V}(\Mat_{1\times n}(\mathbb{F}_q[t]))\big)\bigg)\subset\Mat_{1\times m}(\mathbb{F}_q[t])
    \]
    shows that $\Phi$ induces a matrix $\Phi_{V,W}\in\Mat_{m\times n}(\mathbb{F}_q[t])$. This identification can be extended to the higher dimensional situation as follows. Let $\Phi\in\Mat_{d}(L[t,\tau])$ with
    \[
        \Phi\big(\Mat_{d\times 1}(V\otimes_{\mathbb{F}_q}\mathbb{F}_q[t])\big)\subset\Mat_{d\times 1}(W\otimes_{\mathbb{F}_q}\mathbb{F}_q[t]).
    \]
    By a similar discussion as above, $\Phi$ induces a matrix $\Phi_{V,W}(d)\in\Mat_{md\times nd}(\mathbb{F}_q[t])$.

    The main result obtained in this section is the following statement, which provides an explicit way to compute the solution space $\mathrm{Sol}_\phi(P_1,\dots,P_\ell)$.

    \begin{theorem}\label{Thm:Induced_Linear_System}
        Let $L$ be a finite extension over $K$, $E=(\mathbb{G}_a^d,\phi)$ be a $d$-dimensional $t$-module defined over $L$ and $P_1,\dots,P_\ell\in E(L)$. Then there is an explicit constructed matrix $\mathbb{B}$ with entries in $\mathbb{F}_q[t]$ and $\deg_t(\mathbb{B})\leq 1$ such that we have an explicit surjective $\mathbb{F}_q[t]$-linear map
        \[
            \Ker(\mathbb{B})\twoheadrightarrow\mathrm{Sol}_\phi(P_1,\dots,P_\ell).
        \]
    \end{theorem}

    \begin{proof}
        By Theorem~\ref{Thm:Sol_Divisor}, there is an explicitly constructed effective divisor $D_{\mathcal{P}}\in\mathrm{Div}(L)$ with
        \[
            \mathrm{pr}_1\big(\mathrm{Sol}_\phi(P_1,\dots,P_\ell)\big)\subset\Mat_{d\times 1}(\mathcal{L}(D_{\mathcal{P}})\otimes_{\mathbb{F}_q}\mathbb{F}_q[t]).
        \]
        Since $\mathcal{L}(D_{\mathcal{P}})$ is a finite dimensional $\mathbb{F}_q$-vector space, $\Mat_{d\times 1}(\mathcal{L}(D_{\mathcal{P}})\otimes_{\mathbb{F}_q}\mathbb{F}_q[t])$ is a free $\mathbb{F}_q[t]$-module of finite rank. It follows that $\mathrm{pr}_1\big(\mathrm{Sol}_\phi(P_1,\dots,P_\ell)\big)$ is also a free $\mathbb{F}_q[t]$-module of finite rank. In particular, there is a finite dimensional $\mathbb{F}_q$-vector space $V\subset L$ with $n:=\dim_{\mathbb{F}_q}V\leq\dim_{\mathbb{F}_q}\mathcal{L}(D_{\mathcal{P}})$ such that $\mathrm{pr}_1\big(\mathrm{Sol}_\phi(P_1,\dots,P_\ell)\big)=\Mat_{d\times 1}(V\otimes_{\mathbb{F}_q}\mathbb{F}_q[t])$.
        
        On the other hand, we may consider the divisor of poles $D_{\bm{y}}:=\mathrm{div}(P_1,\dots,P_\ell)_\infty$ as defined in \eqref{E:divinfty} for the matrix $(P_1,\dots,P_\ell)\in\Mat_{d\times\ell}(L)$.
        Then for any $a_1,\dots,a_\ell\in\mathbb{F}_q[t]$, we have 
        \[
            a_1P_1+\cdots+a_\ell P_\ell\in\Mat_{d\times 1}\big(\mathcal{L}(D_{\bm{y}})\otimes_{\mathbb{F}_q}\mathbb{F}_q[t]\big).
        \]
        Consider $\Phi:=\phi_t-t$ and set $m:=\dim_{\mathbb{F}_q}\mathcal{L}(D_{\bm{y}})$. For any $\bm{x}\in\mathrm{pr}_1\big(\mathrm{Sol}_\phi(P_1,\dots,P_\ell)\big)$ we have $\Phi\bm{x}=(\phi_t-t)\bm{x}\in\Mat_{d\times 1}\big(\mathcal{L}(D_{\bm{y}})\otimes_{\mathbb{F}_q}\mathbb{F}_q[t]\big)$. It follows that
        \[
            \Phi\bigg(\mathrm{pr}_1\big(\mathrm{Sol}_\phi(P_1,\dots,P_\ell)\big)\bigg)=\Phi\big(\Mat_{d\times 1}(V\otimes_{\mathbb{F}_q}\mathbb{F}_q[t])\big)\subset\Mat_{d\times 1}\big(\mathcal{L}(D_{\bm{y}})\otimes_{\mathbb{F}_q}\mathbb{F}_q[t]\big).
        \]
        If we fix $\mathbb{F}_q$-basis for $V$ and $\mathcal{L}(D_{\bm{y}})$, then as we discuss above, there is an induced matrix $\Phi_{V,\mathcal{L}(D_{\bm{y}})}\in\Mat_{md\times nd}(\mathbb{F}_q[t])$ such that for any $(\bm{x},a_1,\dots,a_\ell)\in\mathrm{Sol}_\phi(P_1,\dots,P_\ell)$ we have
        \begin{equation}\label{Eq:Induced_Relation_1}
            \iota_{V}(\bm{x})\Phi_{V,\mathcal{L}(D_{\bm{y}})}=a_1\iota_{\mathcal{L}(D_{\bm{y}})}(P_1)+\cdots+a_\ell\iota_{\mathcal{L}(D_{\bm{y}})}(P_\ell).
        \end{equation}
        We express
        \[
            \iota_V(\bm{x})=(x_{11},\dots,x_{1n},x_{21},\dots,\dots, x_{dn})\in\Mat_{1\times nd}(\mathbb{F}_q[t])
        \]
        and
        \[
            \iota_{\mathcal{L}(D_{\bm{y}})}(P_i)=(P^{[i]}_{11},\dots,P^{[i]}_{1m},P^{[i]}_{21},\dots,P^{[i]}_{dm})\in\Mat_{1\times md}(\mathbb{F}_q)
        \]
        for each $1\leq i\leq \ell$. Then \eqref{Eq:Induced_Relation_1} becomes
        \begin{equation}\label{Eq:Induced_Relation_2}
            (x_{11},\dots,x_{1n},x_{21},\dots,x_{dn})\Phi_{V,\mathcal{L}(D_{\bm{y}})}=\big(\sum_{i=1}^\ell a_iP^{[i]}_{11},\dots,\sum_{i=1}^\ell a_iP^{[i]}_{dm}\big).
        \end{equation}
        In particular, \eqref{Eq:Induced_Relation_2} can be regarded as a system of linear relations in $x_{ij}$ and $a_i$ over $\mathbb{F}_q[t]$. It totally has $nd+\ell$ parameters and $md$ relations. Thus, it induces a matrix $\mathbb{B}\in\Mat_{md\times(nd+\ell)}(\mathbb{F}_q[t])$, which gives the desired result.
    \end{proof}

    \begin{example}
        We continue with Example~\ref{Ex:3rdCTs_Pts}. For simplicity, we assume that $q>4$. In this case, if $\bm{x}=(x_1,x_2,x_3)^\tr\in\Mat_{3\times 1}(K[t])$ satisfies
        \[
            \begin{pmatrix}
                \theta-t & 1 & \\
                 & \theta-t & 1\\
                 \tau & & \theta-t
            \end{pmatrix}\begin{pmatrix}
                x_1\\
                x_2\\
                x_3
            \end{pmatrix}=a_1\begin{pmatrix}
                0\\
                0\\
                1
            \end{pmatrix}+a_2\begin{pmatrix}
                0\\
                0\\
                \theta
            \end{pmatrix}
        \]
        for some $a_1,a_2\in\mathbb{F}_q[t]$, then we must have $x_1,x_2,x_3\in A[t]$ and $\deg_\theta(x_i)<i$ for $i=1,2,3$. Thus, we may write
        \[
            \begin{cases}
                x_1=x_1^{[0]}\\
                x_2=x_2^{[0]}+x_2^{[1]}\theta\\
                x_3=x_3^{[0]}+x_3^{[1]}\theta+x_3^{[2]}\theta^2
            \end{cases}
        \]
        for some $x_i^{[j]}\in\mathbb{F}_q[t]$. It follows that
        \begin{equation}\label{Eq:}
            \begin{cases}
                (-tx_1^{[0]}+x_2^{[0]})+(x_1^{[0]}+x_2^{[1]})\theta=0\\
                (-tx_2^{[0]}+x_3^{[0]})+(x_2^{[0]}-tx_2^{[1]}+x_3^{[1]})\theta+(x_2^{[1]}+x_3^{[2]})\theta^2=0\\
                (x_1^{[0]}-tx_3^{[0]})+(x_3^{[0]}-tx_3^{[1]})\theta+(x_3^{[1]}-tx_3^{[2]})\theta^2+x_3^{[2]}\theta^3=a_1+a_2\theta
            \end{cases}.
        \end{equation}
        Then the induced matrix $\mathbb{B}$ is given by
        \[
            \mathbb{B}=\begin{pmatrix}
                -t & 1 & & & & & & \\
                1 & & 1 & & & & & \\
                 & -t & & 1 & & & & \\
                 & 1 & -t & & 1 & & \\
                 & & 1 & & & 1 & & \\
                 1 & & & -t & & & -1 & \\
                 & & & 1 & -t & & & -1 \\
                 & & & & 1 & -t & & \\
                 & & & & & 1 & & &
            \end{pmatrix}\in\Mat_{9\times 8}(\mathbb{F}_q[t])
        \]
        which is equivalent to the system of linear equations
        \[
            \mathbb{B}\begin{pmatrix}
                x_1^{[0]}\\
                x_2^{[0]}\\
                x_2^{[1]}\\
                x_3^{[0]}\\
                x_3^{[1]}\\
                x_3^{[2]}\\
                a_1\\
                a_2
            \end{pmatrix}=\begin{pmatrix}
                0\\
                0\\
                0\\
                0\\
                0\\
                0\\
                0\\
                0\\
                0
            \end{pmatrix}.
        \]
        A straightforward calculation shows that $\Ker(\mathbb{B})$ contains only the zero element. Then Theorem~\ref{Thm:Induced_Linear_System} implies that $\mathrm{Sol}_{\mathbf{C}^{\otimes 3}}(P_1,P_2)$ also contains only the trivial solution. Finally, by Theorem~\ref{Thm:SpPoly}(1), we conclude that $P_1=(0,0,1)^\tr,P_2=(0,0,\theta)^\tr\in\mathbf{C}^{\otimes 3}(K)$ are $\mathbb{F}_q[t]$-linearly independent.
    \end{example}

\section{Algorithms and examples}
    Our algorithms on computing the relation module can be described as follows. We have implemented them in \SageMath \cite{Sage}. Our package is available on GitHub \cite{CHN26}.

\begin{algorithm}[H]
\caption{Compute the lower triangular matrix $A_D$}\label{algorithm_upper}
\KwIn{A $d$-dimensional $t$-module $E=(\mathbb{G}_a^d,\phi)$ defined over $L$.}
\KwOut{A matrix $\mathrm{J}\in\GL_d(L)$ and a lower triangular matrix $A_D\in \Mat_d(L[t,\tau])$ such that, after setting $\widetilde{\phi}_t:=\mathrm{J}\phi_t\mathrm{J}^{-1}$, the matrix $A_D(\widetilde{\phi}_t-t)$ is upper triangular with all diagonal entries nonzero, and $\mathrm{J}(\partial\phi_t)\mathrm{J}^{-1}$ is in Jordan canonical form.}

Choose $\mathrm{J}\in \GL_d(L) $ such that
\[
\mathrm{J}(\partial\phi_t)\mathrm{J}^{-1}
\]
is in Jordan canonical form\footnote{Such a $\mathrm{J}$ exists since all eigenvalues of $\partial \phi_t$ are equal to $\theta$.}\;

Set $\widetilde{\phi}_t:=\mathrm{J}\phi_t\mathrm{J}^{-1}$\;

Set $B:=\widetilde{\phi}_t-t$, $r:=\deg_{\tau}(B)$, and $A_D:=\Id_d$\;

\For{$k=2,\dots,d$}{
    Set $s_k:=r(k-1)$\;

    Write
    \[
    a_{kj}=\sum_{n=0}^{s_k} a_{kj}^{[n]}\tau^n,
    \]
    where $a_{kj}^{[n]}\in L[t]$\;

    Construct the matrix $\widetilde{B}_k\in \Mat_{(r+s_k+1)(k-1)\times (s_k+1)k}(L[t])$ corresponding to the system $(\mathbb{E}_k)$ in the proof of Theorem~\ref{Thm:Triangularization}; equivalently, $\widetilde{B}_k$ encodes the condition that the first $k-1$ entries of
    \[
    (a_{k1},\dots,a_{kk})B
    \]
    vanish\;

    Compute the kernel $\Ker(\widetilde{B}_k)$ over $L[t]$\;

    Choose a vector in $\Ker(\widetilde{B}_k)$ for which $a_{kk}\neq 0$, and reconstruct from it the row
    \[
    (a_{k1},\dots,a_{kk})\in \Mat_{1\times k}(L[t,\tau]).
    \]

    Place this row into the $k$-th row of $A_D$\;
}

\Return{$\mathrm{J}$ and $A_D$}\;
\end{algorithm}

\begin{algorithm}[H]
\caption{Compute $\mathbb{B}$ and $\mathbb{F}_q[t]$-linear relations}\label{algorithm_B}
\KwIn{
\begin{itemize}
    \item a $d$-dimensional $t$-module $E=(\mathbb{G}_a^d,\phi)$ over $L$,
    \item algebraic points $P_1,\dots,P_\ell \in E(L)$,
    \item Matrices $\mathrm{J}$ and $A_D$ obtained from Algorithm \ref{algorithm_upper}.
\end{itemize}
}
\KwOut{A matrix $\mathbb{B}$ over $\mathbb{F}_q[t]$ as in Theorem \ref{Thm:Induced_Linear_System}, the space $\Ker(\mathbb{B})$, and $\mathbb{F}_q[t]$-linear relations among $P_1,\dots,P_\ell$.}

Set $\widetilde{\phi}_t:=\mathrm{J}\phi_t\mathrm{J}^{-1}$, $\widetilde{P}_i=\mathrm{J}P_i$\;

Set $\Phi:=\widetilde{\phi}_t-t$ and $\mathcal{A}:=A_D(\widetilde{\phi}_t-t)$\;

Compute $\mathscr{D}_i$ for $1\le i\le d$ from the upper triangular system
\[
\mathcal{A}\bm{x}=A_D(a_1\widetilde{P}_1+\cdots+a_\ell \widetilde{P}_\ell)
\] as in Lemma~\ref{Lem:Inductive_Bounds_New}\;

\For{$i=1,\dots,d$}{
    Compute a basis of $\mathcal{L}(\mathscr{D}_i)$\;
}

Compute $D_{\bm{y}}:=\operatorname{div}(\widetilde{P}_1,\dots,\widetilde{P}_\ell)_{\infty}$\;

Write $\Phi=(\Phi_{ij})\in \Mat_{d\times d}(L[t,\tau])$\;

Compute $\Phi_{ij}(\mathscr{D}_j)$ for all $i,j$ as defined in \eqref{Eq:divAD}\;

Compute
\[
D_{\Phi\bm{x}}
:=
\bigvee_{i,j}\Phi_{ij}(\mathscr{D}_j);
\]

Compute a basis of $\mathcal{L}(D_{\Phi\bm{x}}\vee D_{\bm{y}})$\;

Introduce the unknowns $x_i\in \mathcal{L}(\mathscr{D}_i)\otimes_{\mathbb{F}_q}\mathbb{F}_q[t]$ for $1\le i\le d$, and $a_1,\dots,a_\ell\in \mathbb{F}_q[t]$\;

Expand the equation
\[
\Phi\bm{x}=\sum_{s=1}^{\ell} a_s\widetilde{P}_s
\]
in the chosen $\mathbb{F}_q$-basis of $\mathcal{L}(D_{\Phi\bm{x}}\vee D_{\bm{y}})$ to obtain a matrix $\mathbb{B}\in \Mat_{m\times n}(\mathbb{F}_q[t])$, where
\[
m\leq d\cdot\dim_{\mathbb{F}_q}\mathcal{L}(D_{\Phi\bm{x}}\vee D_{\bm{y}}),
\qquad
n=\sum_{i=1}^d \dim_{\mathbb{F}_q}\mathcal{L}(\mathscr{D}_i)+\ell.
\]

Compute $\Ker(\mathbb{B})$\;

Project $\Ker(\mathbb{B})$ onto its last $\ell$ coordinates to obtain tuples $(a_1,\dots,a_\ell)\in \mathbb{F}_q[t]^\ell$ satisfying $\sum_{s=1}^{\ell}\widetilde{\phi}_{a_s}(\widetilde{P}_s)=0$, equivalently, $\sum_{s=1}^{\ell}\phi_{a_s}(P_s)=0$\;

\Return{$\mathbb{B}$, $\Ker(\mathbb{B})$, and the $\mathbb{F}_q[t]$-linear relations}\;
\end{algorithm}

\subsection{Additional Example}

\begin{example}Let $L=K(c)$, where $c^{q-1}\in K^{\times}$. Consider the $t$-module $\phi$ such that
\[
\phi_t=
\begin{pmatrix}
\theta & c^{-1}\\
0 & \theta+\tau
\end{pmatrix}.
\]
By Algorithm~\ref{algorithm_upper} or computing directly, we obtain 
\[\mathrm{J}= \begin{pmatrix}
    c & 0 \\ 0 & 1\end{pmatrix},\quad A_D= \Id_2.\]

For
\[
q=5, \quad c^{q-1}=\frac{1}{\theta}+2,
\]
we illustrate Algorithm~\ref{algorithm_B} through the following two tasks:
\begin{enumerate}
\item verify the point
\[
P:=\begin{pmatrix}\theta\\ 0\end{pmatrix}\in \Mat_{2\times 1}(L)
\]
is not an $\FF_q[t]$-torsion point of $\phi$;
\item detect the relation \[
-(t^2+1)P_1+tP_2=0.
\]
between the points
\[
P_1:=\begin{pmatrix}\theta\\ 0\end{pmatrix},
\qquad
P_2:=\begin{pmatrix}\theta^2+1\\ 0\end{pmatrix}
\in \Mat_{2\times 1}(L).
\]
\end{enumerate}
\end{example}

\begin{proof}[Task (1)]We set \[\widetilde{\phi}_t=\mathrm{J}\phi_t\mathrm{J}^{-1}=\begin{pmatrix}
    \theta & 1\\
    0 & \theta+\tau
\end{pmatrix},\quad \widetilde{P}=\mathrm{J}P=\begin{pmatrix}
    \theta c\\
    0 
\end{pmatrix}.\]
First, we consider \[\mathcal{A}\bm{x}=aQ,\] where $ \mathcal{A}=A_D(\widetilde{\phi}_t-t)$, $Q = A_D\widetilde{P}$, $\bm{x}=(x_1,x_2)^\tr\in \Mat_{2\times 1}(L[t])$, and $a\in \FF_q[t]$. We follow lemmas in \S 3.3 to find $\mathscr{D}_i\in\Div(L)$ such that $x_i\in\mathcal{L}(\mathscr{D}_i)\otimes \FF_q[t].$ Note that 
\[\mathcal{A}=\begin{pmatrix}
    \theta-t & 1\\
    0 & (\theta-t)+\tau
\end{pmatrix},\quad Q =\begin{pmatrix}
\theta c\\
0
\end{pmatrix}.\]
Let the places $(g_i(c))\in \Div(L)$ correspond to the decomposition
\[
c^{q-1}-2=\prod_i g_i(c)
\]
over $\FF_q(c)$, and let $(c)$ denote the place corresponding to $c=0$, and let $(\infty_L)$ denote the infinite place. Consider the divisors of poles
\begin{align*}
D_1 =\divisor(\theta c)_\infty&=\divisor\left(\frac{c}{c^{q-1}-2}\right)_\infty\\
&=\left((c)-\sum_i(g_i(c))+(q-2)(\infty_L)\right)_\infty\\
&=\sum_i (g_i(c)),
\end{align*}
and 
\[D_2=\divisor\left(0\right)_\infty=0.\]
Then we have $y_i\in\cL(D_i)\otimes \FF_q[t]$ if we write $Q = (y_1,y_2)^\tr$. 

Suppose $\mathscr{D}_2=\sum o_\nu\cdot (\nu)$. By Lemma \ref{Lem:Inductive_Bounds_New}, we have $\mathscr{D}_2=(D_2)_{\mathcal{A}_{22}}$ where $\mathcal{A}_{22}=(\theta-t)+\tau$. The relevant divisors of $\mathcal{A}_{22}$ are
\[
\divisor(\theta-t)=-\sum_i (g_i(c)),\quad \divisor(1)=0.\]
By Lemma \ref{Lem:1D_Case_New} ,  
when $\nu=(g_i(c))$,
\[
o_\nu=\max\left\{
-\left\lceil\frac{-1-0}{q-1}\right\rceil,
-\left\lceil\frac{0-0}{q}\right\rceil,
0
\right\}=0.
\]
Thus
\[
\mathscr{D}_2=0.
\]
Next, by Lemma \ref{Lem:Inductive_Bounds_New} again,
\[
\mathscr{D}_1=(D_1\vee \mathcal{A}_{12}(\mathscr{D}_2)\bigr)_{\mathcal{A}_{11}}.
\]
Since
\[
\mathcal{A}_{12}=1,
\qquad
\mathcal{D}_2=0,
\]
by Lemma \ref{Lem:1D_Case_New_2} we get
\[
\mathcal{A}_{12}(\mathcal{D}_2)=0.
\]
Therefore
\[
D_1\vee \mathcal{A}_{12}(\mathcal{D}_2)=\sum_i (g_i(c)).
\]
Moreover,
\[
\mathcal{A}_{11}=\theta-t,
\qquad
\divisor(\theta-t)=-\sum_i (g_i(c)).
\]
By Lemma \ref{Lem:1D_Case_New_2}, at each $(g_i(c))$,
\[
\max\left\{
-\left\lceil\frac{-(-1)-(-1)}{q^0}\right\rceil,
0
\right\}=0.
\]
Hence
\[
\mathscr{D}_1=0.
\]
In summary, $\mathscr{D}_1=\mathscr{D}_2=0.$

Secondly, we consider 
\[\Phi\bm{x}=\bm{y}\]
where $\Phi=\widetilde{\phi}_t-t$, and $\bm{y}=a\widetilde{P}$. We continue on the steps in Algorithm \ref{algorithm_B} to find the matrix $\mathbb{B}$ as in Theorem \ref{Thm:Induced_Linear_System}. We have
\[
\widetilde{P}=\begin{pmatrix}\theta c\\ 0\end{pmatrix}
,
\qquad
\Phi=
\begin{pmatrix}
\theta-t & 1\\
0 & (\theta-t)+\tau
\end{pmatrix}.
\]
The relevant divisors are
\[
D_{\bm{y}}=\divisor(\widetilde{P})_{\infty}=0,
\qquad
D_{\Phi \bm{x}}=\bigvee_{i,j}\Phi_{ij}(\mathscr{D}_j)=\sum_i (g_i(c)).
\]
Hence
\[
 D_{\Phi {\bm{x}}}\vee D_{\bm{y}}
=\sum_i (g_i(c)).
\]
For $q=5$,
\[
\dim \cL(D_{\Phi {\bm{x}}}\vee D_{\bm{y}})=\deg \left(D_{\Phi {\bm{x}}}\vee D_{\bm{y}}\right)+1=5.
\]
By \SageMath, an $\FF_q$-basis of $\cL(D_{\Phi {\bm{x}}}\vee D_{\bm{y}})$ is
\[
\{\theta,1,\theta c,\theta c^2, \theta c^3\}.
\]
Now we express the equations in $\Phi\bm{x}=\bm{y}$ into $\FF_q[t]$-linear combinations in $\cL(D_{\Phi {\bm{x}}}\vee D_{\bm{y}})\otimes \FF_q[t]$ of the chosen basis. We have 
\[
\begin{pmatrix}
\theta-t & 1\\
0 & (\theta-t)+\tau
\end{pmatrix}
\begin{pmatrix}x_1\\ x_2\end{pmatrix}
= a\begin{pmatrix}\theta c\\ 0\end{pmatrix},
\]
where
\[
x_1,x_2\in \cL(0)\otimes_{\FF_q}\FF_q[t]=\FF_q[t],
\]
i.e., $x_i=x_i^{[0]}$.
Hence
\[
\begin{cases}
(\theta-t)x_1+x_2=a\theta c,\\
(\theta-t)x_2+x_2=0.
\end{cases}
\]
Equivalently,
\[
\begin{cases}
x_1\cdot \theta+(-tx_1+x_2)\cdot 1 + (-a)\cdot \theta c=0,\\
x_2\cdot \theta+(-t+1)x_2\cdot 1=0.
\end{cases}
\]
Thus the corresponding linear system is
\[\mathbb{B}
\begin{pmatrix}
x_1\\ x_2\\ a
\end{pmatrix}=
\begin{pmatrix}
1 & 0 & 0\\
-t & 1 & 0\\
0 & 0 & -1\\
0 & 1 & 0\\
0 & -t+1 & 0
\end{pmatrix}
\begin{pmatrix}
x_1\\ x_2\\ a
\end{pmatrix}=0.
\]
It is not hard to see that $\nul(\mathbb{B})=0$, which implies that the algebraic point $P=\begin{pmatrix}
\theta\\ 0
\end{pmatrix}\in E(L)$ is not $\FF_q[t]$-torsion by Theorem \ref{Thm:Induced_Linear_System}.
    
\end{proof}

\begin{proof}[Task (2)] Similar to Task (1), 
we set \[\widetilde{\phi}_t=\mathrm{J}\phi_t\mathrm{J}^{-1}=\begin{pmatrix}
    \theta & 1\\
    0 & \theta+\tau
\end{pmatrix},\quad \widetilde{P}_1=\mathrm{J}P_1=\begin{pmatrix}
    \theta c\\
    0 
\end{pmatrix},\quad \widetilde{P}_2=\mathrm{J}P_2=\begin{pmatrix}
    (\theta^2+1) c\\
    0 
\end{pmatrix}.\]
First, we consider \[\mathcal{A}\bm{x}=a_1Q+a_2Q_2,\] where $ \mathcal{A}=A_D(\widetilde{\phi}_t-t)$, $Q_i = A_D\widetilde{P}_i$, $\bm{x}=(x_1,x_2)^\tr\in \Mat_{2\times 1}(L[t])$, and $a\in \FF_q[t]$. By the same steps as Task 1, one can find 
\begin{align*}
D_1 &= \divisor(\theta c, (\theta^2+1)c)_\infty =2\sum_i(g_i(c))+(\infty_L),\\
D_2 &= 0,\\
\mathscr{D}_2 &=(D_2)_{\mathcal{A}_{22}}=0,\\
\mathscr{D}_1&=(D_1\vee \mathcal{A}_{12}(\mathscr{D}_2)\bigr)_{\mathcal{A}_{11}}=\sum_i(g_i(c))+(\infty_L).
\end{align*}
Then we have $x_i\in\mathcal{L}(\mathscr{D}_i)\otimes \FF_q[t].$

Secondly, we consider 
\[\Phi\bm{x}=\bm{y}\]
where $\Phi=\widetilde{\phi}_t-t$, and $\bm{y}=a_1\widetilde{P}_1+a_2\widetilde{P}_2$. For these data, one can check the divisor 
\[
 D_{\Phi {\bm{x}}}\vee D_{\bm{y}}
=2\sum_i(g_i(c))+(\infty_L).
\]
By \SageMath, for $q=5$, we obtain $\FF_q$-basis of the following Riemann-Roch spaces:
\begin{align*}
\cL(D_{\Phi {\bm{x}}}\vee D_{\bm{y}}): &\quad\{\theta^2, \theta, 1, \theta^2 c, \theta c, c, \theta^2 c^2, \theta c^2, \theta^2 c^3, \theta c^3\},\\
\cL(\mathscr{D}_1):&\quad\{\theta, 1, \theta c, c, \theta c^2, \theta c^3\},\\
\cL(\mathscr{D}_2):&\quad\{1\}.
\end{align*}
Now we express the equations in $\Phi\bm{x}=\bm{y}$ into $\FF_q[t]$-linear combinations in $\cL(D_{\Phi {\bm{x}}}\vee D_{\bm{y}})\otimes \FF_q[t]$ of the chosen basis. Expanding $x_i$ as follows:
\begin{align*}
x_1 &=x_1^{[0]}\cdot\theta+x_1^{[1]}\cdot 1+x_1^{[2]}\cdot \theta c+x_1^{[3]}\cdot c+x_1^{[4]}\cdot \theta c^2+x_1^{[5]}\cdot \theta c^3,\\
x_2&=x_2^{[0]},
\end{align*}
the system 
\[
\begin{pmatrix}
\theta-t & 1\\
0 & (\theta-t)+\tau
\end{pmatrix}
\begin{pmatrix}x_1\\ x_2\end{pmatrix}
= a_1\begin{pmatrix}\theta c\\ 0\end{pmatrix}+a_2\begin{pmatrix}(\theta^2+1) c\\ 0\end{pmatrix}
\]
gives
\[
\begin{cases}
\begin{aligned}
& x_1^{[0]}\cdot \theta^2 + (-tx_1^{[0]}+x_1^{[1]})\cdot\theta + (-tx_1^{[1]}+x_2)\cdot 1 + (x_1^{[2]}-a_2)\cdot\theta^2c \\
& \quad + (-tx_1^{[2]}+x_1^{[3]}-a_1)\cdot \theta c + (-tx_1^{[3]}+a_2)\cdot c + x_1^{[4]}\cdot \theta^2c^2 \\
& \quad + (-t)x_1^{[4]}\cdot \theta c^2 + x_1^{[5]}\cdot\theta^2c^3 + (-tx_1^{[5]})\cdot \theta c^3 = 0, \\[1ex]
& x_2\cdot \theta + (-t+1)x_2\cdot 1 = 0.
\end{aligned}
\end{cases}
\]
Thus the corresponding matrix $\mathbb{B}$ in Theroem \ref{Thm:Induced_Linear_System} is
\[\mathbb{B}
=\begin{pmatrix} 1 & 0 & 0 & 0 & 0 & 0 & 0 & 0 & 0  \\ 
-t & 1 & 0 & 0 & 0 & 0 & 0 & 0 & 0  \\
0 & -t & 0 & 0 & 0 & 0 & 1 & 0 & 0  \\
0 & 0 & 1 & 0 & 0 & 0 & 0 & 0 & -1  \\
0 & 0 & -t & 1 & 0 & 0 & 0 & -1 & 0  \\
0 & 0 & 0 & -t & 0 & 0 & 0 & 0 & 1  \\
0 & 0 & 0 & 0 & 1 & 0 & 0 & 0 & 0  \\
0 & 0 & 0 & 0 & -t & 0 & 0 & 0 & 0  \\
0 & 0 & 0 & 0 & 0 & 1 & 0 & 0 & 0  \\
0 & 0 & 0 & 0 & 0 & -t & 0 & 0 & 0 \\
0 & 0 & 0 & 0 & 0 & 0 & 1 & 0 & 0 \\
0 & 0 & 0 & 0 & 0 & 0 & -t+1 & 0 & 0 
\end{pmatrix}.
\]
One can check $\nul(\mathbb{B})=\mathrm{Span}_{\FF_q[t]}\left\{(0,0,t,-1,0,0,0,-t^2-1,t)^\tr\right\}$, where the last $2$ coordinates provide the relation \[
-(t^2+1)P_1+tP_2=0.
\]
between the points
$
P_1:=\begin{pmatrix}\theta\\ 0\end{pmatrix},
P_2:=\begin{pmatrix}\theta^2+1\\ 0\end{pmatrix}
\in \Mat_{2\times 1}(L).
$
\end{proof}

\subsection{Applications to Carlitz multiple polylogarithms}
    Let $r\geq 0$ be a positive integer. Following \cite[Def.~5.1.1]{Cha14}, for each $r$-tuple $\fs=(s_1,\dots,s_r)\in\mathbb{Z}_{>0}^r$ of positive integers, we define the Carlitz multiple polylogarithm (CMPL) as
    \[
        \Li_{\fs}(z_1,\dots,z_r):=\sum_{i_1>i_2>\cdots>i_r\geq 0}\frac{z_1^{q^{i_1}}\cdots z_r^{q^{i_r}}}{L_{i_1}^{s_1}\cdots L_{i_1}^{s_1}}\in K\llbracket z_1,\dots,z_r\rrbracket
    \]
    where $L_0:=1$ and $L_i:=(\theta-\theta^q)\cdots(\theta-\theta^{q^i})\in\bA$ for each $i\geq 1$. Note that the CMPL converges on
    \[
        \mathbb{D}_{\fs}:=\{(u_1,\dots,u_r)\in\mathbb{C}_\infty^r\mid|\frac{u_1^{q^{i_1}}}{L_{i_1}^{s_1}}|_\infty\cdots|\frac{u_r^{q^{i_r}}}{L_{i_r}^{s_r}}|_\infty\to 0~\mathrm{as}~0\leq i_1\leq\cdots\leq i_r\to\infty\}.
    \]
    For a finite extension $L/K$ with $L\subset\mathbb{C}_\infty$, we set $\mathbb{D}_\fs(L):=\mathbb{D}_\fs\cap (L^\times)^r$. For $\bu=(u_1,\dots,u_r)\in\mathbb{D}_\fs(L)$, the value $\Li_\fs(\bu)$ is called \emph{Eulerian} if $\Li_\fs(\bu)/\Tilde{\pi}^{\wt(\fs)}\in K$, where $\wt(\fs):=s_1+\cdots+s_r$ and $\tilde{\pi}:=\theta (-\theta)^{1/(q-1)}\prod_{i=1}^{\infty} (1-\theta^{1-q^i})^{-1}\in \mathbb{C}_{\infty}^{\times}$ for some fixed $(q-1)$th root of $-\theta$. 

    For $1 \leq \ell \leq r$,
we set $d_{\ell} := s_{\ell} + \cdots + s_{r}$ and $d := d_{1} + \cdots + d_{r}$.
Let $B$ be a $d \times d$-matrix of the form

\[
\left( \begin{array}{c|c|c}
B[11] & \cdots & B[1r] \\ \hline
\vdots & & \vdots \\ \hline
B[r1] & \cdots & B[rr]
\end{array} \right)
\]
where $B[\ell m]$ is a $d_{\ell} \times d_{m}$-matrix for each $\ell$ and $m$.
We call $B[\ell m]$ the $(\ell, m)$-th block sub-matrix of $B$.

For $1 \leq \ell \leq m \leq r$, we set

\[
N_{\ell} := \left(
\begin{array}{ccccc}
0 & 1 & 0 & \cdots & 0 \\
& 0 & 1 & \ddots & \vdots \\
& & \ddots & \ddots & 0 \\
& & & \ddots & 1 \\
& & & & 0
\end{array}
\right)
\in \Mat_{d_{\ell}}(K),
\]

\[
N := \left(
\begin{array}{cccc}
N_{1} & & & \\
& N_{2} & & \\
& & \ddots & \\
& & & N_{r}
\end{array}
\right)
\in \Mat_{d}(K),
\]

\[
E[\ell m] := \left(
\begin{array}{cccc}
0 & \cdots & \cdots & 0 \\
\vdots & \ddots & & \vdots \\
0 & & \ddots & \vdots \\
1 & 0 & \cdots & 0
\end{array}
\right)
\in \Mat_{d_{\ell} \times d_{m}}(K) \ \ \ (\mathrm{if} \ \ell = m),
\]

\[
E[\ell m] := \left(
\begin{array}{cccc}
0 & \cdots & \cdots & 0 \\
\vdots & \ddots & & \vdots \\
0 & & \ddots & \vdots \\
(-1)^{m-\ell} \prod_{e=\ell}^{m-1} u_{e} & 0 & \cdots & 0
\end{array}
\right)
\in \Mat_{d_{\ell} \times d_{m}}(L) \ \ \ (\mathrm{if} \ \ell < m),
\]

\[
E := \left(
\begin{array}{cccc}
E[11] & E[12] & \cdots & E[1r] \\
& E[22] & \ddots & \vdots \\
& & \ddots & E[r-1,r] \\
& & & E[rr]
\end{array}
\right)
\in \Mat_{d}(L).
\]

Finally, we define the $t$-module $G_{\fs, \bu} := (\GG_{a}^{d}, \rho)$ by
\begin{equation}\label{E:Explicit t-moduleCMPL}
  \rho_{t} = \theta I_{d} + N + E \tau
  \in \Mat_{d}(L[\tau]).
\end{equation}
Note that $G_{\fs,\bu}$ depends  only on $u_{1},\ldots,u_{r-1}$.

Let
\begin{equation}\label{E:v_s,u}
    \bv_{\fs, \bu} := \begin{pmatrix}
    0\\
    \vdots\\
 0 \\
(-1)^{r-1} u_{1} \cdots u_{r} \\
0 \\
 \vdots  \\
 0  \\
 (-1)^{r-2} u_{2} \cdots u_{r} \\
 \vdots \\
 0  \\
 \vdots  \\
 0  \\
 u_{r}
    \end{pmatrix}
    \begin{aligned}
&\left.\begin{matrix}
\\
\\
\\
\\
\end{matrix} \right\} %
d_1\\ %
&\left.\begin{matrix}
\\
\\
\\
\\
\end{matrix} \right\} %
d_2\\ %
&\begin{matrix}
\\
\end{matrix} \, \, \vdots\\ %
&\left.\begin{matrix}
\\
\\
\\
\\
\end{matrix}\right\}%
d_r\\
\end{aligned}  \in G_{\fs,\bu}(L).
\end{equation}

Note that the above $t$-module was only theoretically
constructed in \cite{CPY19}, but was explicitly written down in \cite[\S3.1]{CM19}.
The application of our algorithm to CMPLs at algebraic points is the following result.
\begin{theorem}
    For $\bu\in\mathbb{D}_\fs(L)$, there is an algorithm determining the Eulerianess of $\Li_\fs(\bu)$.
\end{theorem}

\begin{remark}
    Note that if $L=K$, then it has been shown in \cite[Cor.~6.1.2]{CPY19} that there is an effective criterion for deciding if $\Li_\fs(\bu)$ is Eulerian or not. The result there is based on the knowledge of the structure of rational torsions on the tensor powers of the Carlitz module \cite[Lem.~5.1.3]{CPY19}
    \[
        \mathbf{C}^{\otimes d}(K)_{\mathrm{tor}}:=\{\bm{x}\in\Mat_{d\times 1}(K)\mid [a]_d(\bm{x})=0~\mathrm{for~some}~a\in\mathbb{F}_q[t]\}.
    \]
    Our algorithm presented below provides a different approach that works for arbitrary global function field $L$.
\end{remark}

\begin{proof}
    For $\bu\in\mathbb{D}_\fs(L)$, by \cite[Thm.~4.3.2(a)]{CPY19} (cf. \cite[Thm.~5.1.2,~Cor.~5.1.3]{CM19}), the value $\Li_\fs(\bu)$ is Eulerian if and only if $\bv_{\fs,\bu}\in G_{\fs,\bu}(L)$ is a torsion point. For $1\leq i\leq r$, we denote by $G_i$ the $t$-module whose $t$-action comes from cutting the upper left square of size $d_1+\cdots+d_i$ from the matrix defined in \eqref{E:Explicit t-moduleCMPL}. Then we have the short exact sequence of $t$-modules for each $1\leq i\leq r$
    \[
        0\to G_{i-1}\to G_i\overset{\pi_i}{\to}\mathbf{C}^{\otimes d_i}\to 0.
    \]
    Since $G_r=G_{\fs,\bu}$, we can define $\bv_r:=\bv_{\fs,\bu}$. It follows that $\pi_r(\bv_r)\in\mathbf{C}^{\otimes d_r}(L)$. Now, we can perform Algorithm~\ref{algorithm_upper} and Algorithm~\ref{algorithm_B} to compute the relation module of $\pi_r(\bv_r)\in\Mat_{d_r\times 1}(L)$ on the $t$-module $\mathbf{C}^{\otimes d_r}$. If the relation module is trivial, then $\bv_{\fs,\bu}$ can not be a torsion element in $G_{\fs,\bu}(L)$ and we are done. Otherwise, the relation module must be a cyclic submodule since $\mathbb{F}_q[t]$ is a principal ideal domain. Assume that it is generated by $a_r\in\mathbb{F}_q[t]$. Then $\pi_r\big(\rho_{a_r}(\bv_r)\big)=[a_r]_{d_r}\big(\pi_r(\bv_r)\big)=0$
    shows that $\bv_{r-1}:=\rho_{a_r}(\bv_r)\in\Ker\pi_r=G_{r-1}(L)$.

    Consequently, the inductive process above either terminates if the relation module of $\pi_i(\bv_i)\in\mathbf{C}^{\otimes d_i}(L)$ is trivial, or we find a sequence of non-zero elements $\{a_i\}_{i=0}^r\subset\mathbb{F}_q[t]$ such that if we define $\mathfrak{a}:=\prod_{i=0}^ra_i\in\mathbb{F}_q[t]$ then $\rho_{\mathfrak{a}}(\bv_{\fs,\bu})=0$. In the latter case, $\bv_{\fs,\bu}$ is a torsion element on $G_{\fs,\bu}$ which implies that $\Li_\fs(\bu)$ is Eulerian. The desired result now follows.
\end{proof}

\bibliographystyle{alpha}

\end{document}